\pdfoutput=1
 \documentclass[11pt,leqno]{amsart}
\usepackage{amssymb}
\usepackage{amsmath}
\usepackage{amsxtra}
\usepackage{amscd}
\usepackage{mathtools}
\usepackage{float}
\usepackage{graphicx} 
\usepackage{amsfonts}
\usepackage[utf8]{inputenc}
\usepackage[T1]{fontenc}
\usepackage{enumerate}
\usepackage{color}
\usepackage[all]{xy}
\usepackage[bookmarks=false]{hyperref}
\usepackage{amsbsy}
\usepackage{transparent}
\allowdisplaybreaks

\newtheorem{theorem}{Theorem}
\newtheorem{corollary}{Corollary}

\newtheorem{proposition}{Proposition}
\newtheorem{definition}{Definition}
\newtheorem{remark}{Remark}

\numberwithin{equation}{section}

\makeatletter
\let\@wraptoccontribs\wraptoccontribs
\makeatother

\begin{document} 

\title[Classic link invariants from ${\rm Y}_{\MakeLowercase{d,n}}(\MakeLowercase{q})$ and ${\rm FTL}_{\MakeLowercase{d,n}}(\MakeLowercase{q})$]{Classical link invariants from the framizations of the Iwahori-Hecke algebra and the Temperley-Lieb algebra of type $A$}

\author{Dimos Goundaroulis}
\address{Department of Mathematics,
National Technical University of Athens,
Zografou campus, GR-157 80 Athens, Greece.}
\email{dgound@mail.ntua.gr}
\urladdr{users.ntua.gr/dgound}

\author{Sofia Lambropoulou}
\address{Department of Mathematics,
National Technical University of Athens,
Zografou campus, GR-157 80 Athens, Greece.}
\email{sofia@math.ntua.gr}
\urladdr{www.math.ntua.gr/~sofia/}

\thanks{The first author has been financed by the ``European Solidarity Travel Grant'' of the European Mathematical Society. The second author has been financed by the European Union (European Social
 Fund - ESF) and Greek national funds through the Operational Program
 "Education and Lifelong Learning" of the National Strategic Reference
 Framework (NSRF) - Research Funding Program: THALES: Reinforcement of the
 interdisciplinary and/or inter-institutional research and innovation.
}

\keywords{Framization, Iwahori-Hecke algebra, Temperley-Lieb algebra, Ocneanu trace, Yokonuma-Hecke algebra, Markov trace, link invariants}

\subjclass[2010]{57M25, 57M27, 20C08, 20F36}

\date{}

\begin{abstract}
In this paper we first present the construction of the new 2-variable classical link invariants arising from the Yokonuma-Hecke algebras ${\rm Y}_{d,n}(q)$, which are not topologically equivalent to the Homflypt polynomial. We then present the algebra ${\rm FTL}_{d,n}(q)$ which is the appropriate Temperley-Lieb analogue of ${\rm Y}_{d,n}(q)$, as well as the related 1-variable classical link invariants, which in turn are not topologically equivalent to the Jones polynomial. Finally, we present the algebra of braids and ties which is related to the Yokonuma-Hecke algebra, and also its quotient, the partition Temperley-Lieb algebra ${\rm PTL}_n(q)$ and we prove an isomorphism of this algebra with a subalgebra of ${\rm FTL}_{d,n}(q)$.
\end {abstract}
\maketitle

\section{Introduction}
In the 80's, the Jones polynomial \cite{jo} and its 2-variable generalization, the Homflypt polynomial \cite{Homfly,pt}, rekindled the interest in the study of mathematical knots by providing strong and easy-to-compute machinery for distinguishing pairs of non-isotopic classical knots and links. The methods of V. F. R. Jones \cite{jo} to construct  both of these polynomial invariants involved, for the first time in the literature, Artin's braid group $B_n$, a representation of $B_n$ onto either the Temperley-Lieb algebra ${\rm TL}_n(q)$ or the Iwahori-Hecke algebra ${\rm H}_n(q)$ respectively, and the use of a unique linear Markov trace defined on these algebras. Thus, the algebras ${\rm H}_n(q)$ and ${\rm TL}_n(q)$ became the first examples of knot algebras. A knot algebra ${\rm A}$ is a triplet $({\rm A}, \pi , \tau ) $, where $\pi$ is an appropriate representation of the braid group in ${\rm A}$ and $\tau$ a Markov trace function defined on ${\rm A}$.

\smallbreak
Aiming at discovering new 3-manifold invariants J. Juyumaya and the second author introduced the concept of {\it framization of knot algebras} \cite{jula2, jula5, jula, jula6}. The framization consists in an extension of  a knot algebra via the  addition of  framing generators which gives rise to a new algebra that is related to framed braids and framed knots.  The basic example of framization is the Yokonuma-Hecke algebra, ${\rm Y}_{d,n}(q)$, which can be regarded as a  framization of the Iwahori-Hecke algebra, ${\rm H}_n(q)$ \cite{jula2, jula}. Juyumaya in \cite{ju} constructed a unique Markov trace function, ${\rm tr}_{d}$, on the algebra ${\rm Y}_{d,n}(q)$ with parameters $z, x_1 , \ldots , x_{d-1}$, hence making ${\rm Y}_{d,n}(q)$ a knot algebra. One of the challenges that appeared along the way of constructing framed and classical link invariants  from the algebras ${\rm Y}_{d,n}(q)$ was the fact that ${\rm tr}_{d}$ does not re-scale according to the framed braid equivalence, making it the only trace known in the literature  with this property \cite{jula}. However, by applying the so-called {\it ${\rm E}$-condition} on the parameters $x_1, \ldots , x_{d-1}$ the trace ${\rm tr}_{d}$ does re-scale. The solutions of the ${\rm E}$-system were determined by P. G\'erardin \cite[Appendix]{jula}.

\smallbreak
 The resulting invariants, in particular those for classical links, was necessary to be compared with other known invariants, especially with the 2-variable Jones or Homflypt polynomial. Indeed, while it was already known \cite{ChLa} that the polynomial invariants in question do not coincide with the Homflypt polynomial except in trivial cases, they could still be topologically equivalent, in the sense that they might distinguish the same pairs of non-isotopic links. This problem remained open for quite some time, until in \cite{ChJuKaLa} it was proven that the classical link invariants $\Theta_d$ from the algebras ${\rm Y}_{d,n}(q)$  coincide with the Homflypt polynomial on {\it knots}, but they are {\it not topologically equivalent to the Homflypt polynomial on links}. The proof emerged from the discovery of a {\it special skein relation} for the invariants $\Theta_d$ {\it involving only crossings between different components} of the link. This fact is very important, since there are very few link invariants defined through skein relations. The intrinsic reason behind this discovery was the quadratic relation that was used in \cite{ChJuKaLa}. In all previous works regarding the Yokonuma-Hecke algebras \cite{ju, jula2, jula3, jula4, jula, jula5, jula6}, another presentation was used with a parameter $u$ in a different quadratic relation, where $u=q^2$. The new quadratic relation revealed the special skein relation and also simplified computations significantly in the software that was used in the comparison of the invariants \cite{ka}. In \cite{ChJuKaLa}, 6 pairs of Homflypt-equivalent links are presented, which are distinguished by the invariants $\Theta_d$. Regarding properties, these invariants behave similarly to the Homflypt polynomial under reversing orientation, split links, connected sums and mirror imaging \cite{ChmJaKaLa, ChJuKaLa}.

\smallbreak
The next natural question  was the determination of the framization of the Temperley-Lieb algebra. In \cite{go, gojukola, gojukola2} potential candidate quotients of the Yokonuma-Hecke algebra were introduced and studied extensively. There were three possible quotients of the algebra ${\rm Y}_{d,n}(u)$ that could lead to a framization of the Temperley-Lieb algebra, the {\it Yokonuma-Temperley-Lieb algebra} ${\rm YTL}_{d,n}(u)$, the {\it Complex Reflection Temperley-Lieb algebra} ${\rm CTL}_{d,n}(u)$ and the {\it Framization of the Temperley-Lieb algebra} ${\rm FTL}_{d,n}(u)$. In ${\rm YTL}_{d,n}(u)$ the defining ideal is generated by an element analogous to the Steinberg element of the classical Temperley-Lieb algebra. In ${\rm FTL}_{d,n}(u)$ framing is introduced intrinsically in the defining element of the ideal. Finally, in ${\rm CTL}_{d,n}(u)$ framing in the defining element is less restricting than in ${\rm FTL}_{d,n}(u)$. Next, the necessary and sufficient conditions so that the Juyumaya trace ${\rm tr}_{d}$ passes through to each one of the three quotient algebras had to be determined in order that they qualify as knot algebras (together with the natural representation of the framed braid group onto each one of them). The corresponding conditions for the algebra ${\rm YTL}_{d,n}(u)$ are given in \cite{go, gojukola}, however they are too restrictive, and as a result, the related classical link invariants just recover the Jones polynomial. For this reason the algebra ${\rm YTL}_{d,n}(u)$ was discarded as a potential candidate for the framization of the Temperley-Lieb algebra. The conditions for the algebras ${\rm CTL}_{d,n}(u)$ and ${\rm FTL}_{d,n}(u)$ were explored in \cite{go, gojukola2}. For the case of ${\rm CTL}_{d,n}(u)$, contrary to the case of ${\rm YTL}_{d,n}(u)$, the conditions are too relaxed, leading to the necessity of imposing the ${\rm E}$-condition on the trace parameters $x_1, \ldots ,x_{n-1}$  in order to obtain link invariants. Even by doing so, the resulting invariants coincide  with those that are derived either from ${\rm Y}_{d,n}(u)$ or from ${\rm FTL}_{d,n}(u)$ \cite{gojukola2}. Consequently, as was discussed in \cite{go, gojukola2}, the most natural candidate from the topological point of view was to choose the algebra ${\rm FTL}_{d,n}(u)$ as the framization of the Temperley-Lieb algebra, since the necessary and sufficient conditions for the passing of the trace include the ${\rm E}$-condition. Focusing now on the classical link invariants from the algebra ${\rm FTL}_{d,n}(u)$, following the methods of \cite{ChJuKaLa}, these needed to be compared to the Jones polynomial. To achieve this, we considered in \cite{gojukola2} a new presentation for the algebra with parameter $q$ instead of $u$ that had to be adopted from the algebra ${\rm Y}_{d,n}(q)$. Denoting now the related classical link invariants by $\theta_d$ and  adjusting the results of \cite{ChJuKaLa} to the invariants $\theta_d(q)$, it was shown \cite{gojukola2} that they coincide with the Jones polynomial on {\it knots} but they are {\it not topologically equivalent to the Jones polynomial on links}.

\smallbreak
Framizations of other knot algebras have been also proposed. For example, the framization of the BMW algebra is introduced  in \cite{jula5}, while \cite{ChdA, flojula} discuss the framization of Hecke-type algebras of type $B$.

\smallbreak
Returning now to the invariants from the algebra ${\rm Y}_{d,n}(q)$, it was shown in \cite{ChJuKaLa} that the family of invariants $\{ \Theta_d \}$ can generalized to a 3-variable invariant $\Theta(q,\lambda, E)$. The invariant $\Theta$ can be completely defined using just the special skein relation of $\Theta_d$ and its values on disjoint unions of knots \cite{ChJuKaLa, kaula}. Further, the invariant $\Theta$ is related to the {\it algebra of braids and ties}, $\mathcal{E}_n(u)$, that was introduced by F. Aicardi and J. Juyumaya in \cite{AiJu}. In this paper we will use a different presentation for the algebra of braids and ties, with parameter $q$ instead of $u$, that was first given in \cite{ChJuKaLa}. We will use this presentation to define a Temperley-Lieb type quotient of the algebra $\mathcal{E}_n(q)$, the {\it partition Temperley-Lieb algebra}, ${\rm PTL}_n(q)$, which was originally defined by J. Juyumaya in \cite{juptl} as a quotient of the algebra $\mathcal{E}_n(u)$ and we show that for $d\geq n$ it is isomorphic to the subalgebra of ${\rm FTL}_{d,n}(q)$ that is generated only by the braiding generators $g_i$. In a similar way to the invariant $\Theta (q,\lambda , E)$, the algebra ${\rm PTL}_n(q)$ is related to a 2-variable generalization of the classical link invariants $\theta_d$ from the algebras ${\rm FTL}_{d,n}(q)$ \cite{gola}.

In this paper we give a survey of the results in \cite{ChJuKaLa, gojukola, gojukola2}. Furthermore, we point out the connection of the partition Temperley-Lieb algebra with a certain subalgebra of the Framization of the Temperley-Lieb algebra, which is a new result. 

\smallbreak
The outline of the paper is as follows:
Section~\ref{prelim} is dedicated to providing necessary definitions and results, including: the Iwahori-Hecke algebra, the Ocneanu trace, the Homflypt polynomial, the Temperley-Lieb algebra and the Jones polynomial. In Section~\ref{frhecke} we recall some basic facts for the framed braid group and we give the definition of the Yokonuma-Hecke algebra. In Section~\ref{sectinv}, using  tools from harmonic analysis on finite groups, such as the convolution product, the product by coordinates and the Fourier transform, we give a proof of the solutions of the {\it ${\rm E}$-system} using the notation introduced in this paper. Then, we describe the construction of the invariants $\Phi_{d}$ for framed links and $\Theta_{d}$ for classical links. In Section~\ref{yhinvsec} we demonstrate the methods of \cite{ChJuKaLa} for comparing the invariants $\Theta_{d}$ to the Homflypt polynomial by employing the {\it specialized trace} ${\rm tr}_{d,D}$. In Section~\ref{frtl}, we discuss the results of \cite{gojukola2} regarding the 1-variable invariants for classical links derived from the Framization of the Temperley-Lieb algebra ${\rm FTL}_{d,n}(q)$ and how they are proven to be {\it not topologically equivalent to the Jones polynomial for the case of links.} Finally, in Section~\ref{furth} we give the connection to the partition Temperley-Lieb algebra ${\rm PTL}_{d,n}(q)$ by proving the existence for $d \geq n$ of an isomorphism between ${\rm PTL}_{d,n}(q)$ and the subalgebra of ${\rm FTL}_{d,n}(q)$ that is generated only by the braiding generators $g_i$.

\smallbreak
This paper is an extended version, including new results, of the talk with title ``On the link invariants associated to the framization of knot algebras'' that was given by the first author at the Special Session 35, {\it "Low Dimensional Topology and Its Relationships with Physics", } as part of  the 1st International AMS/EMS/SPM Meeting held at Porto June 10-13 2015.

\section{Preliminaries}\label{prelim}
\subsection{{\it Notations}}Throughout the paper by the term algebra we mean an associative unital algebra over the field $\mathbb{C}(q)$, where $q$ is an indeterminate. Two positive integers, $d$ and $n$, are also fixed. 

We now introduce the groups that will be used in the paper. We denote by $S_n$ the symmetric group on $n$ symbols and by $s_i$ the elementary transposition $(i,i+1)$. 

Denote now by $C = \langle t \rangle$ the infinite cyclic group and by $C_d= \langle t \, | \, t^d=1 \rangle$  the cyclic group of order $d$. Let $t_i := (1,\ldots, 1,t ,1, \ldots , 1 )$,
where $t$ is in the  $i$-th position. We then define:
\[
C^n := \langle t_1, \ldots  ,t_n \, | \, t_i t_j = t_j t_i, \quad \forall \, i, \, j \rangle \qquad \mbox{and} \qquad
C_d^n:= C^n / \langle t_i^d -1  \rangle .
\]

We further define the group $C_{d,n}: = C_d^n \rtimes S_n $, where the action is defined by permutation on the indices of the $t_i$'s, namely: $s_it_j = t_{s_i(j)} s_i$.  Notice that $C_{d,n}$ is isomorphic to the {\it complex reflection group} $G(d,1,n)$. Finally, the {\it braid group of type} $A$, denoted by $B_n$, is the group generated by the elementary braidings $\sigma_1, \ldots , \sigma_{n-1}$, subject to the braid relations: $\sigma_i \sigma_j \sigma_i = \sigma_j \sigma_i \sigma_j$, for $|i-j|=1$ and $\sigma_i \sigma_j = \sigma_j \sigma_i$, for $|i-j|>1$.

\subsection{{\it The Iwahori-Hecke algebra of type $A$ and the Homflypt polynomial}}\label{homflysection}
The Iwahori-Hecke algebra of type $A$, ${\rm H}_n(q)$, is  the  $\mathbb{C}(q)$-algebra that is generated by the elements $h_1, \ldots , h_{n-1}$ that are subject to the following relations:
 \begin{eqnarray}
h_i h_j &=& h_j h_i \quad \text{for all} \quad |i-j| >1 \label{He1}\\
h_i h_j h_i &=& h_j h_i h_j \quad \text{for all} \quad \vert i - j \vert =1 \label{He2} \\
h_i^2 &=&   1 + (q - q^{-1}) h_i \label{He3}.
\end{eqnarray}
The first two relations in the presentation of ${\rm H}_n(q)$ are exactly the braid relations. Thus, there exists a natural epimorphism $ \pi: \mathbb{C}(q)B_n \to {\rm H}_n(q)$, that sends $\sigma_i \mapsto h_i$ and, so, ${\rm H}_n(q)$ can be also viewed as the quotient of $\mathbb{C}(q)B_n$ over the two-sided ideal generated by the quadratic relations \eqref{He3}. Relations \eqref{He3} imply that the generators $h_i$ are invertible. Indeed:
\begin{equation}\label{hinv}
h_i^{-1} = (q^{-1} -q ) + h_i .
\end{equation}

\begin{remark}\label{heckrem} \rm
Alternatively, the algebra ${\rm H}_n(q)$ can be seen as a $q$-deformation of the group algebra $\mathbb{C}S_n$, namely as  the  $\mathbb{C}(q)$-algebra that is generated by the elements  $h_w$, where  $w\in S_n$ and the following rules of multiplication:
\[
h_{s_i}h_w =\left\{\begin{array}{ll}
h_{s_iw} & \text{for } l(s_iw)>l(w) \\
  h_{s_iw} + (q-q^{-1})h_w & \text{for } l(s_iw)<l(w),
\end{array}\right.
\]
where $l$ denotes the length function on $S_n$ with respect to the generators $s_i$. Setting now $h_i := h_{s_i} $, one obtains the presentation that is described by \eqref{He1}-\eqref{He3}.
\end{remark}

One of the most important properties of the Iwahori-Hecke algebra, is that it supports a unique Markov trace function which was first proved by Ocneanu  \cite{jo, Homfly}. Namely, for any indeterminate $z$ there exists a linear trace ${\rm \tau}$ on $\cup_{n=1}^{\infty} {\rm H}_n(q)$ uniquely defined by the inductive rules:
\[
\begin{array}{lllll}
(1)& {\rm \tau} (\mathbf{1}_{n+1})&= 1, &\mbox{for all } n&\\
(2)& {\rm \tau} (a b) &= {\rm \tau} (ba), & a,b \in {\rm H}_n(q) &(\mbox{Conjugation property})\\
(3)&{\rm \tau} (a h_n ) &= z \, {\rm \tau}(a), & a\in {\rm H}_n(q) &(\mbox{Markov property}),
\end{array}
\]
where $\mathbf{1}_{n+1}$ denotes the unit in ${\rm H}_{n+1}(q)$. By using the natural epimorphism $\pi$ and by abusing notation, one can define $\tau$ on the elements of $B_n$. From the topological point of view, closing a braid $\alpha$, that is, connecting corresponding end points in pairs, gives rise to an oriented link. The closed braid is denoted by $\widehat{\alpha}$ and is called {\it the closure of the braid} $\alpha$. For the converse, by Alexander's theorem \cite{alex}, any oriented link is isotopic to the closure of a braid. Further, by the well-known Markov theorem \cite{markov}, isotopy classes of oriented links are in bijection with equivalence classes of braids in $\cup_{\infty}B_n$. The equivalence relation is generated by the two following Markov moves:
\begin{enumerate}[i.]
\item Conjugation: $\alpha \beta \sim \beta \alpha$, $\alpha$, $\beta \in B_n$.
\item Stabilization move (positive and negative): $\alpha \sim \alpha \sigma_n^{\pm1}$, $\alpha \in B_n$.
\end{enumerate}

V. F. R. Jones used the Markov theorem and the Ocneanu trace for defining an invariant for knots and links in the context of braids by arguing that the trace $\tau$ must satisfy both Markov moves \cite{jo}. By the second rule of  the trace $\tau$, we deduce that $\tau$ already satisfies the conjugation move. However, the third defining rule of $\tau$ ensures only positive stabilization. Therefore, $\tau$ must be first  re-scaled so that the generators $h_i$ and $h_i^{-1}$ yield the same trace value. Let $\lambda_{\rm H} \in \mathbb{C}(q)$ such that:
\begin{equation}\label{lamdh}
\tau (\sqrt{\lambda_{\rm H}}h_i) = \tau \left ( ( \sqrt{\lambda_{\rm H}} h_i)^{-1} \right).
\end{equation}
Using equation~\eqref{hinv} one finds:
\begin{equation}\label{theta}
\lambda_{\rm H} = \frac{ \tau \left( h_i^{-1} \right)}{\tau(h_i)} = \frac{z + q^{-1} - q}{z}.
\end{equation}
Equation~\eqref{lamdh}  ensures that $\tau (w h_n) = \tau (w h_n^{-1} )$ for any $w \in {\rm H}_n(q)$. This is in accordance with the fact that the links $\widehat{\alpha \sigma_n}$ and $\widehat{\alpha \sigma_n^{-1}}$ are isotopic.
Furthermore, since the links $\widehat{\alpha}$ and $\widehat{\alpha \sigma_n}$ are isotopic, the trace $\tau$ must be normalized so that the link invariant takes the same value on them. This can be achieved by introducing the {\it normalization factor}  $  \frac{1 -\lambda_{\rm H} }{\sqrt{\lambda_{\rm H}}(q-q^{-1})} $. We can now proceed with the definition of the {\it 2-variable Jones or Homflypt polynomial} \cite{jo}:
\begin{definition}\label{xinvdef} \rm
The two-variable invariant $P(q,\lambda_{\rm H})$ of the oriented link $\widehat{\alpha}$ is the function:
\begin{equation}\label{xinv}
P(q,\lambda_{\rm H})(\widehat{\alpha}) = \left ( \frac{1 -\lambda_{\rm H} }{\sqrt{\lambda_{\rm H}}(q-q^{-1})} \right )^{n-1} \left ( \sqrt{\lambda_{\rm H}} \right )^{\epsilon(\alpha)} \tau ( \pi ( \alpha )) ,
\end{equation}
where $\alpha \in B_n$, $\epsilon(\alpha)$ is the algebraic sum of the exponents of the $\sigma_i$'s in $\alpha$ and $ \pi$ is the natural epimorphism from $\mathbb{C}(q)B_n$ to  ${\rm H}_n(q)$.
\end{definition} 
Moreover, the polynomial $P(q,\lambda_{\rm H})$ satisfies the following skein relation:
\[
\frac{1}{\sqrt{\lambda_{\rm H}}} \, P({L_{+}}) - \, \sqrt{\lambda_{\rm H}} \,  P({L_{-}}) =  \left( q - q^{-1} \right )P({L_0}),
\]
where $L_{+}, L_{-}$ and $L_0$ constitute a Conway triple \cite{jo}.

\subsection{\it The Temperley-Lieb algebra and the Jones polynomial} For $n \geq 3$, the classical {\it Temperley-Lieb algebra} ${\rm TL}_n(q)$ can be defined  as the quotient of the algebra ${\rm H}_n(q)$ over the two-sided ideal generated by the {\it Steinberg  elements}:
\begin{equation}\label{idealrel}
h_{i,j}: = 1 + q(h_i + h_j) + q^2 (h_i h_j + h_j h_i) + q^3 h_i h_j h_i,\quad \text{for all} \quad \vert i - j \vert =1 .
\end{equation}
The defining ideal of the algebra ${\rm TL}_n(q)$ is principal and it is generated by the element $h_{1,2}$ \cite[Corollary 2.11.2]{gohajo}. Furthermore, using the transformation:
\begin{equation}\label{transfor}
f_i := \frac{1}{q^2+1}(qh_i + 1) ,
\end{equation}
the algebra ${\rm TL}_n(q)$ can be presented by the non-invertible generators $ f_1, \ldots , f_{n-1} $ subject to the following relations:
\[
\begin{array}{ccll}
f_{i}^{2} & = &  f_i & \\
f_if_{j}f_i &  = &  \delta f_i & \text{for all} \quad |i-j| =1\\
f_if_j &  = & f_j f_i & \text{for all} \quad |i-j| >1 ,
\end{array}
\]
where $ \delta^{-1} = 2+ q^2 +q^{-2}$ \cite{jo}.

The Ocneanu trace $\tau$ passes through to the quotient algebra ${\rm TL}_n(q)$ for specific values of $z$. Indeed, as V. F. R. Jones showed in \cite{jo}, to factorize $\tau$ to the Temperley-Lieb algebra, one only needs the requirement that $\tau$ annihilates the expressions \eqref{idealrel}. So, it is proved in \cite{jo} that  
$\tau$ factors through the algebra ${\rm TL}_n(q)$ if and only if the trace indeterminate $z$ takes the values:
\begin{equation}\label{jonval}
z = - \frac{q^{-1}}{q^2+1} \quad \mbox{or} \quad z=-q^{-1}.
\end{equation}

By specializing $z$ to $-\frac{q^{-1}}{q^2+1}$, which is the only topologically non-trivial value for $z$, one obtains the Jones polynomial, $V(q)$, through the Homflypt polynomial, as follows  \cite{jo}:
\begin{equation}\label{vpol}
V(q)(\widehat{\alpha}) = \left(- \frac{1+q^2}{q} \right)^{n-1} q^{2\,\varepsilon(\alpha)} {\rm \tau}(\pi(\alpha)) = P(q,q^4)(\widehat{\alpha}),
\end{equation}
where $\alpha \in B_n$ and $\epsilon(\alpha)$, $\pi$ are as in \eqref{xinv}.
\section{The framization of the Iwahori-Hecke algebra of type $A$}\label{frhecke}
In this section we give the definition of the Yokonuma-Hecke algebras, as quotients of the framed braid groups, and we also recall a Markov trace defined on them.
\subsection{{\it The $d$-modular framed braid group}}
The {\it framed braid group} on $n$ strands is defined as:
\begin{equation*}
{\mathcal F}_{n} = C^n \rtimes  B_n ,
\end{equation*}
where the action of $B_n$ on $C^n$ is given by the permutation induced by a braid on the indices:
$
\sigma_it_j=t_{s_i(j)}\sigma_i.
$
The generators $t_i$ of $C^n$ are called the {\it framing generators}, since $t_i$ means framing 1 on the $i$-th strand of a braid. Due to the above action a word $w$ in ${\mathcal F}_{n}$ has the {\it splitting property}, that is, it splits into the  {\it framing} part and the {\it braiding}  part:
$
w = t_1^{a_1}\ldots t_n^{a_n} \, \sigma,
$
where $\sigma \in B_n$ and $a_i \in \mathbb{Z}$, and thus $w$ is a classical braid with an integer attached to each
strand. Topologically, an element of $C^n$ is identified with a framed identity braid on $n$ strands, while a classical braid in $B_n$ is viewed as a framed braid with all framings zero. The multiplication in ${\mathcal F}_n$ is defined by placing one framed braid on top of the other and collecting the total framing of each strand to the top.

For the fixed positive integer $d$, the {\it $d$-modular framed braid
group} on $n$ strands, ${\mathcal F}_{d,n}$, is defined as the
quotient of ${\mathcal F}_n$ over the {\it modular relations}
$
t_i^d= 1$, for $ i =1, \ldots, n
$
and thus, ${\mathcal F}_{d,n}= C_d^n \rtimes  B_n$. Framed braids in ${\mathcal F}_{d,n}$  have framings modulo $d$. 

\subsection{{\it The Yokonuma-Hecke algebra of type A}}
The {\it Yokonuma-Hecke algebra} ${\rm Y}_{d,n}(q)$ \cite{yo} is defined as the quotient of the group algebra
$\mathbb{C}(q) {\mathcal F}_{d,n}$ over the two-sided ideal  generated by the elements:
\[
\sigma_i^2 - 1 -  (q-q^{-1}) \, e_i \, \sigma_i, \quad \mbox{for all } i,
\]
where   $e_{i}$ is (the idempotent \cite{ju}) defined by $e_{i} := \frac{1}{d} \sum_{s=0}^{d-1}t_i^s t_{i+1}^{d-s},
$ where $i=1,\ldots , n-1$ (see \cite{jula} for a diagrammatic interpretation). The generators of the ideal give rise to the following quadratic relations in ${\rm Y}_{d,n}(q)$: 
\begin{equation}\label{sigmaquadr}
g_i^2 = 1 + (q-q^{-1}) \, e_i \, g_i,
\end{equation}
where $g_i$ corresponds to $\sigma_i$. Since the quadratic relations do not change the framings we have that $\mathbb{C}C_d^n\subset \mathbb{C}(q) C_d^n \subset  {\rm Y}_{d,n}(q)$ and we keep the same notation for the elements of $\mathbb{C}C_d^n$ and for the elements $e_i$  in $ {\rm Y}_{d,n}(q)$. From the above, the algebra ${\rm Y}_{d,n}(q)$ has a presentation with generators $ t_1, \ldots , t_n , g_1,$ $ \ldots, g_{n-1}$, subject to the following relations:
\begin{align}
g_i g_j &= g_j g_i \quad | i - j| >1 \label{y1} \\
 g_i g_j g_i& = g_j g_i g_j \quad |i-j|=1 \label{y2}\\
 t_i^d&=1 \label{y3}\\
 t_i t_j &= t_j t_i \label{y4}\\
 g_i t_j& = t_{s_i(j)} g_i \label{y5}\\
 g_i^2 = 1+ (&q - q^{-1})\,e_i \, g_i \label{yqeq}.
 \end{align}
 The $t_i$'s are the {\it framing generators}, while the $g_i$'s are the {\it braiding generators} of ${\rm Y}_{d,n}(q)$. Note that all generators are invertible. In particular:
 \begin{equation}\label{ginv}
 g_i^{-1} = g_i - (q-q^{-1}) e_i.
 \end{equation}
 
By its construction, the Yokonuma-Hecke algebra of type $A$ is considered as {\it the framization of the Iwahori-Hecke algebra of type} A.

\begin{remark} \label{yhalt} \rm
In analogy to Remark~\ref{heckrem}, the algebra ${\rm Y}_{d,n}(q)$ can be alternatively defined as a $q$-deformation of the algebra $\mathbb{C}C_{d,n}$. Namely, it is the $\mathbb{C}(q)$-algebra that is generated by the elements $g_{s_i}$ and $g_{t_i}$, where  $s_i\in S_n$ and $t_i \in \left ( \mathbb{Z} /d \mathbb{Z} \right )^n$, and the following rules of multiplication: 
\[
g_{s_i}g_w =\left\{\begin{array}{ll}
g_{s_iw} & \text{for } l'(s_iw)>l'(w) \\
 g_{s_iw} + (q-q^{-1})e_ig_w & \text{for } l'(s_iw)<l'(w) ,
\end{array}\right.
\]
where $w \in S_n$ and $l^\prime$ is the length function of $C_{d,n}$. We further define: $g_{t_iw} = g_{t_i}g_w$. Setting now $g_i := g_{s_i}$ and $t_i:= g_{t_i}$ one recovers the presentation described by equations \eqref{y1}-\eqref{yqeq}. The reader may refer to \cite[Proposition~2.14]{jusur} and  \cite{ju, gojukola2} for more details. 
\end{remark}

 \begin{remark}\label{oldquadr} \rm
 In \cite{ju, jula, jula2, jula3, jula4, jula5} a different presentation for the Yokonuma-Hecka algebra, was used, with parameter $u$. More precisely,  the algebra ${\rm Y}_{d,n}(u)$ is generated by the elements $\widetilde{g}_1 , \ldots , \widetilde{g}_{n-1} , t_1 , \ldots , t_n, $ satistfying the relations (\ref{y1})-(\ref{y5})  and the quadratic relations:

\begin{equation}\label{newquad}
(\widetilde{g}_i)^{\,2} = 1+ (u-1) e_i + (u-1) e_i  \, \widetilde{g}_i.
\end{equation}

One can obtain the presentation given above for ${\rm Y}_{d,n}(q)$ from this one by taking $u=q^2$ and 
\begin{equation}\label{switch}
 \widetilde{g}_i=g_i+(q^{-1}-1)e_i \, g_i  \quad \mbox{or, equivalently, }\,g_i= \widetilde{g}_i+(q-1)e_i \, \widetilde{g}_i. 
\end{equation} \end{remark}

\begin{remark} \rm
For $d=1$,  the algebra ${\rm Y}_{1,n}(q)$ coincides with the algebra ${\rm H}_n(q)$. Namely, for $d=1$ we have that $t_j=1$, for $1\leq j \leq n$, which also implies that $e_i=1$, $1\leq i \leq n-1$. Then the quadratic relation \eqref{yqeq} becomes: $g_i^2 = 1+ (q-q^{-1}) g_i$, which is the quadratic relation of the algebra ${\rm H}_n(q)$. On the braid level,  the group $\mathcal{F}_{1,n}$ coincides with the group $B_n$ of classical braids, since for $d=1$ all framings are zero. 
\end{remark}

The idempotents $e_i$ can be generalized  to the following elements in ${\rm Y}_{d,n}(q)$ for any indices $i,j$:
\[
e_{i,j} := \frac{1}{d} \sum_{s=0}^{d-1}t_i^s t_{j}^{d-s}.
\]
We also define, for any $0\leq m \leq d-1$, {\it the shift of $e_i$ by $m$}:
\[
e_i^{(m)} := \frac{1}{d} \sum_{s=0}^{d-1} t_i^{m+s} t_{i+1}^{d-s}.
\]
Notice that $e_i = e_{i, i+1} = e_i^{(0)}$. Notice also that $e_i^{(m)} = t_i^m e_i = t_{i+1}^m e_i$. Then  one deduces easily that: $e_i^{(m)}e_{i+1}  =   e_ie_{i+1}^{(m)} \quad \mbox{for all } \, 0\leq m \leq d-1$.

In \cite{ju} it is shown that the algebra ${\rm Y}_{d,n+1}(q)$ has the following inductive linear basis:
\begin{equation}\label{yhbasel}
\mathfrak{m}_{n}g_n g_{n-1}\ldots g_i t_i^k \quad \mbox{or} \quad  \mathfrak{m}_{n}t_{n+1}^k ,
\end{equation}
where $0\leq k \leq d-1$ and $\mathfrak{m}_n$ is  a word in the inductive basis of  ${\rm Y}_{d,n}(q)$. By employing this basis Juyumaya has proven that ${\rm Y}_{d,n}(q)$ supports a unique Markov trace function for any $n$:
\begin{theorem}[{\cite[Theorem 12]{ju}}]\label{Juyutrace}
For indeterminates $z$, $x_1$, $\ldots, x_{d-1}$ there exists a unique linear Markov trace:
\[
 {\rm tr}_d:  \cup_{n=1}^{\infty}{\rm Y}_{d,n}(q) \longrightarrow    \mathbb{C}(q)[z, x_1, \ldots, x_{d-1}],
\]
 defined inductively on $n$ by the following rules:
\[
\begin{array}{rcll}
{\rm tr}_d(ab) & = & {\rm tr}_d(ba)  \qquad &  \\
{\rm tr}_d(\mathbf{1}_{n+1}) & = & 1 & \\
{\rm tr}_d(ag_n) & = & z\, {\rm tr}_d(a) \qquad &  (\text{Markov  property} )\\
{\rm tr}_d(at_{n+1}^s) & = & x_s {\rm tr}_d(a)\qquad  & (  s = 1, \ldots , d-1) ,
\end{array}
\]
where $\mathbf{1}_{n+1}$ denotes the unit in ${\rm Y}_{d,{n+1}}(q)$ and $a,b \in {\rm Y}_{d,n}(q)$. In analogy to the classical case, by considering the natural epimorphism $\gamma:\mathbb{C}(q)\mathcal{F}_n \longrightarrow {\rm Y}_{d,n}(q)$ that sends $\sigma_i \mapsto g_i$ and $t_i \mapsto t_i$, and by abusing notation, one can define ${\rm tr}_d$ on the elements of $\mathcal{F}_n$.
\end{theorem}

\begin{remark} \rm
Using the trace rules of Theorem~\ref{Juyutrace} and setting  $x_0:=1$, we deduce that ${\rm tr}_d(e_i)$ takes the same value for all $i$, and this value is denoted by $E$:
\[
E := {\rm tr}_d(e_i)= \frac{1}{d}\sum_{s=0}^{d-1}x_{s}x_{d-s}.
\]
Moreover, we also define {\it the shift by $m$ of $E$}, where $0\leq m \leq d-1$, by:
\[
E^{(m)} :={\rm tr}_d(e_i^{(m)})= \frac{1}{d}\sum_{s=0}^{d-1}x_{m+s}x_{d-s} .
\]
Notice that  $E^{(0)} = E$.
\end{remark}

\section{Framed and classical link invariants from ${\rm Y}_{d,n}(q)$}\label{sectinv}
 In complete analogy to the classical case, isotopy classes of oriented framed links are in bijection with equivalence classes of framed braids. This equivalence is generated by the Markov moves, adjusted to the case of framed braids \cite{ks}, namely:
\begin{enumerate}[i.]
\item Conjugation: $\alpha \beta \sim \beta \alpha$, $\alpha$, $\beta \in \mathcal{F}_n$.
\item Stabilization move (positive and negative): $\alpha \sim \alpha \sigma_n^{\pm1}$, $\alpha \in \mathcal{F}_n$.
\end{enumerate}

 \subsection{{\it The ${\rm E}$-system}}\label{invdefsec} The aim now would be to re-scale the trace ${\rm tr}_d$ so that: ${\rm tr}_d(a g_n) = {\rm tr}_d(a g_{n}^{-1})$, $a \, \in \, {\rm Y}_{d,n}(q)$. This way we ensure that ${\rm tr}_d$ satisfies the positive and negative stabilization moves. Moreover, we have to normalize the trace ${\rm tr}_d$ so that links $\widehat{\alpha}$ and $\widehat{\alpha \sigma_n}$, where $\alpha \in \mathcal{F}_{n}$, get the same value of the invariant derived from ${\rm tr}_{d}$. Unfortunately, contrary to the case of $\tau$, the trace ${\rm tr}_d$ does not re-scale directly. As explained in \cite{jula}, the issue lies in the fact that ${\rm tr}_d(a g_n^{-1})$ does not factor through ${\rm tr}_d (a)$. Indeed from \eqref{ginv} we have that: ${\rm tr}_d (a g_n^{-1}) = {\rm tr}_d (a g_n) + (q-q^{-1}) {\rm tr}_d (a  e_n)$, which is not equal to ${\rm tr}_d (g_n^{-1}) {\rm tr}_d(a)$. This is because on the level of braids the framing of the word $a$ changes in $a  e_n$. Thus, ${\rm tr}_d (a e_n)$ does not factor through ${\rm tr}_d(a)$. In order to overcome this obstacle, the trace parameters $x_m$, $1 \leq m \leq d-1$, must be chosen so that they comprise a solution of the following non-linear ${\rm E}$-{\it system} of equations, for any $m \in \mathbb{Z}/d\mathbb{Z}$:
 
 \[
E^{(m)} = x_m\, E,
 \]
 or, equivalently:

\begin{equation}\label{esys}
\sum_{s=0}^{d-1} x_{m+s}\, x_{d-s} = x_m \sum_{s=0}^{d-1} x_s\, x_{d-s},
\end{equation}
where the indices are considered modulo $d$ and $x_0 =1$. An obvious solution of the ${\rm E}$-system is the ``trivial'' solution where all $x_i$'s take the value zero. Another solution easy to come up with is when the $x_i$'s are specialized to the $d$-th roots of unity. The full set of solutions was established by P. G\'erardin by solving the ${\rm E}$-system in $\mathbb{C}C_d$ \cite[Appendix]{jula}. Before continuing, we shall introduce some necessary notation in order to explain G\'erardin's method. Denote by $\delta_a$ the element $t^a$ of the canonical basis of $\mathbb{C}C_d$, where $a\in \mathbb{Z}/d\mathbb{Z}$. An element in the group algebra $\mathbb{C}C_d$ has the form: $\sum_{k=0}^{d-1}a_k t^k$. The {\it product by coordinates} in $\mathbb{C}C_d$ is defined  by the formula: 
\[
\left( \sum_{r=0}^{d-1} a_r t^r\right) \cdot 
\left( \sum_{s=0}^{d-1} b_s t^s\right)=
 \sum_{i=0}^{d-1} a_ib_i t^i
\]
and the  {\it convolution product} is defined by the formula:
\begin{equation} \label{AA1}
\left( \sum_{r=0}^{d-1} a_r t^r\right) \ast  
\left( \sum_{s=0}^{d-1} b_s t^s\right) = 
\sum_{r=0}^{d-1} \left( \sum_{s=0}^{d-1} a_s b_{r-s} \right) t^r .
\end{equation}

Denote now by $\chi_k$ the characters of the group $C_d$, namely: $\chi_k( t^m) = \cos\frac{2\pi km}{d} + i \sin \frac{2 \pi km}{d}$, where $k, m\in {\mathbb{Z}/d\mathbb{Z}}$. Define also the elements $\mathbf{i}_a:=\sum_{s=0}^{d-1} \chi_a( t^s) t^s \, \in \mathbb{C}C_d$, where $a\in {\mathbb Z}/d{\mathbb Z}$. The {\it Fourier transform} is the linear automorphism on $\mathbb{C}C_d$, defined by: 
\begin{equation}\label{fourtrans}
y:=  \sum_{r=0 }^{d-1}a_rt^r \mapsto \widehat{y}:= \sum_{m=0}^{d-1}(y\ast \mathbf{i}_m)(0)t^m  ,
\end{equation}
where $(y\ast \mathbf{i}_s)(0)$ denote the coefficient of $\delta_0$ in the convolution $y\ast \mathbf{i}_s$. For more details on the tools that are used here the reader is referred to \cite{jula, gojukola2}. In the following proposition, we present the most important properties of the Fourier transform that will be used in this paper.
\begin{proposition}[{\cite[Chapter~2]{te}}]\label{propietrans}
For any $y$ and $y^{\prime}$ in $\mathbb{C}C_d$, we have:
\begin{enumerate}
\item $\widehat{y \ast y^{\prime}}=\widehat{y}\cdot \widehat{y^{\prime}}$
\item $\widehat{y\cdot y^{\prime}}= d^{-1}\widehat{y} \ast \widehat{y^{\prime}}$
\item $\widehat{\delta}_a=\mathbf{i}_{-a}$
\item $\widehat{\mathbf{i}}_a=d \delta_a$
\item If $ y = \sum_{r=0}^{d-1}a_rt^r$, then 
$
\widehat{\widehat{y}}=  d \sum_{r=0}^{d-1}a_{d-r}t^r .
$
\end{enumerate}
\end{proposition}

For the solutions of the ${\rm E}$-system we use the notation $({\rm x}_1, \ldots , {\rm x}_{d-1})$. Denoting $x$ the complex function on $C_d$ that that maps $0$ to $1$ and $k$ to ${\rm x}_k$,   we have the following by G\'erardin \cite[Appendix]{jula}, for which we give a detailed proof using the notation introduced in this paper.

\begin{theorem}[G\'erardin]\label{esysol}
The solutions of the ${\rm E}$-system are of the following form:
\[
{\rm x}_s = \frac{1}{|D|} \sum_{m \in D} \mathbf{i}_m (t^s), \quad 1 \leq s \leq d-1 ,
\]
where $x$ as above and $D$ is the support of the Fourier transform of $x$. Hence the solutions of the ${\rm E}$-system are parametrized by the non-empty subsets of $\mathbb{Z}/d\mathbb{Z}$.
\end{theorem}

\begin{proof}
We will solve equation \eqref{esys} in $\mathbb{C}C_d$. Let $x = \sum_{k=0}^{d-1} {\rm x}_k t^k$ $\in \mathbb{C}C_d$ and observe that by using the convolution product in $\mathbb{C}C_d$ we have the following equalities:
\begin{equation}\label{convprod}
x\ast x =  \sum_{s=0}^{d-1} {\rm x}_{m+s} {\rm x}_{d-s} \,t^s  \quad \mbox{and} \quad (x\ast x)(0) = \sum_{s=0}^{d-1} {\rm x}_{s} {\rm x}_{d-s} .
\end{equation}
Combining now equations \eqref{esys} and \eqref{convprod} we obtain the functional equivalent of the ${\rm E}$-system. Indeed, we have that:
\begin{equation}\label{esysfunc}
x\ast x = (x \ast x) (0) \, x.
\end{equation}
Applying the Fourier transform on equation \eqref{esysfunc}, we then obtain:
\begin{equation}\label{esysft}
\widehat{x}^2 = (x \ast x) (0) \, \widehat{x}.
\end{equation}
The case $(x \ast x)(0) =0 $, which leads to $\widehat{x} =0$ and subsequently to $x=0$, is excluded by the  condition $x(0)=1$. Thus, if $\widehat{x} = \sum_{k=0}^{d-1} y_k t^k$, we have that \eqref{esysft} becomes:
\[
\sum_{k=0}^{d-1}y_k^2  t^k = (x \ast x) (0) \sum_{k=0}^{d-1}y_k  t^k
\]
and so we obtain:
\[
y_k = (x\ast x)(0).
\]
Denoting now by $D$ the support of $\widehat{x}$ we have that $D= \{ k \in \mathbb{Z}/d\mathbb{Z} \, | \, y_k = (x \ast x) (0) \}$ and we then deduce that:
\begin{equation}
\widehat{x} = \sum_{m \in D} (x\ast x) (0) \delta_m  ,
\end{equation}
which, using argument (4) of Proposition \ref{propietrans}, yields:
\[
 \widehat{\widehat{x}}=  (x\ast x) (0) \sum_{m\in D}\mathbf{i}_{d-m}.
\]
Having in mind now argument (5) of Proposition \ref{propietrans}, we deduce that:
\begin{equation}\label{dblft}
x=  \frac{1}{d}(x\ast x) (0) \sum_{m\in D} \mathbf{i}_{m} .
 \end{equation}
Since $x(0)=1$, we have that $\frac{1}{d}(x\ast x) (0)\, |D| = 1$, or equivalently, $\frac{1}{d}(x\ast x) (0)= \frac{1}{|D|}$. Therefore, \eqref{dblft} becomes:
\[
x= \frac{1}{|D|} \sum_{m\in D} \mathbf{i}_m
\]
and the proof of the Theorem is concluded.
\end{proof}

\begin{remark} \rm
The solution of the ${\rm E}$-system where the ${\rm x}_i$'s are $d$-th roots of unity is parametrized by the singleton subsets $\{ m \}$ of $ \mathbb{Z}/ d\mathbb{Z}$. On the other hand, the whole set $\mathbb{Z}/d \mathbb{Z}$ parametrizes the trivial solution.
\end{remark}

 \subsection{{\it The specialized trace}}\label{sptrsec}  We now fix throughout a solution $ X_D = ( {\rm x}_1 , \ldots , {\rm x}_{d-1} )$ of the ${\rm E}$-system parametrized by the non-empty subset $D$ of  $\mathbb{Z}/d\mathbb{Z}$.  As discussed above the following is true: 

\begin{theorem}[{\cite[Theorem 7]{jula}}]\label{trmultip}
If the trace parameters  $(x_1, \ldots , x_{d-1} )$ satisfy the  ${\rm E}$-condition,  then for any $a\in {\rm Y}_{d,n}(u)$ we have that:
\[
{\rm tr}_d(a e_n) = {\rm tr}_d(a) {\rm tr}_d(e_n).
\]
\end{theorem}

Subsequently, by Theorem~\ref{trmultip},  ${\rm tr}_d(a g_{n}^{-1})$ can factor through ${\rm tr}_d (a)$, as required \cite{jula}. Consequently, in this case, ${\rm tr}_d$ satisfies both positive and negative stabilization moves for framed braids. 

In order to define framed and classical link invariants using the trace ${\rm tr}_d$, one would like to specialize the parameters $x_i$  to a solution of the ${\rm E}$-system, as early as possible in the construction. For this reason the specialized trace ${\rm tr}_{d,D}$ is introduced in \cite{ChLa}.

\begin{definition}\label{sptrdef} \rm
Let $z$ be an indeterminate. The trace map ${\rm tr}_{d,D}$, defined as the trace ${\rm tr}_d$  with the parameters $x_1, \ldots , x_{d-1}$ specialized to the values ${\rm x}_1 , \ldots , {\rm x}_{d-1}$, shall be called the \emph{specialized trace}  with parameter $z$. 
\end{definition}

\begin{remark} \rm
Following Theorem~\ref{esysol} we have that \cite{jula4}:
\[
E_D:= {\rm tr}_{d,D} (e_i) = \frac{1}{|D|} , \quad \mbox{for all } i.
\]
Moreover, if $|D|=1$ we obtain $E_D =1$, while if $D = \mathbb{Z} / d\mathbb{Z}$ and so ${\rm E}_D = \frac{1}{d}$.
\end{remark}

\begin{remark}\rm
For $d=1$ the specialized trace ${\rm tr}_{1, \{0 \}}$ coincides with ${\rm tr}_{1}$ which, in turn, coincides with the Ocneanu trace $\tau$.
\end{remark}

\subsection{{\it Framed and classical link invariants from ${\rm Y}_{d,n}(q)$}} As mentioned earlier, the above apply naturaly to the construction of topological invariants for framed knots and links. We proceed now with introducing the re-scaling factor for ${\rm tr}_{d,D}$. We set:
\begin{equation}\label{lambdad}
\lambda_D := \frac{z- (q-q^{-1}) E_D}{z}.
\end{equation}

We then have the following:
\begin{theorem}{{\cite[Theorem~3.1]{ChJuKaLa}}}\label{frinvthm}
For any framed braid $\alpha \in \mathcal{F}_n$ the following mapping: 
\begin{equation}\label{phinv}
\Phi_{d,D}(q,\lambda_D)(\widehat{\alpha}) := \left ( \frac{1 - \lambda_D}{\sqrt{\lambda_D} (q-q^{-1})E_D} \right)^{n-1} \left ( \sqrt{\lambda_D} \right)^{\varepsilon(\alpha)} {\rm tr}_{d,D}(\gamma ( \alpha ) ) 
\end{equation}
is a 2-variable invariant of framed oriented links, where $\widehat{\alpha}$ is the closure of the framed braid $\alpha \in \mathcal{F}_n$, $\varepsilon ( \alpha)$ is the algebraic sum of the exponents of the braiding generators $\sigma_i$ in the braid word $\alpha$ and $\gamma$ is the natural epimorphism from $\mathbb{C}(q)\mathcal{F}_n$ to ${\rm Y}_{d,n}(q)$.
\end{theorem}

The invariants $\Phi_{d,D}$ satisfy a skein relation involving the framing and the braiding generators \cite{jula, ChJuKaLa}:
\begin{equation}\label{phiskein}
\frac{1}{\sqrt{\lambda_D}} \Phi_{d,D}(L_{+}) - \sqrt{\lambda_D}\, \Phi_{d,D}(L_{-}) = \frac{q-q^{-1}}{d} \sum_{s=0}^{d-1} \Phi_{d,D} (L_{s}) ,
\end{equation}
 where, for $\beta \in \mathcal{F}_n$, $L_+ = \widehat{\beta \sigma_i}$ , $L_{-} = \widehat{\beta \sigma_i^{-1}}$ and $L_s = \widehat{\beta t_i^{s} t_{i+1}^{-s}}$ are identical links except in one crossing (see Figure~\ref{phils}).
 
 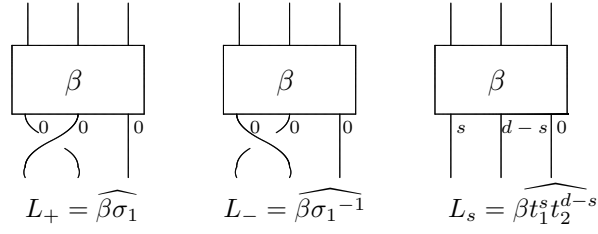
\begin{figure}[th]
 \centering
\scalebox{1}{\begin{picture}(216,80)
\put(20,53){$\beta$}

\qbezier(6,70)(6,78)(6,86) 
\qbezier(25,70)(25,78)(25,86)
\qbezier(44,70)(44,78)(44,86)

\qbezier(5,20)(4,22)(5,24) 
\qbezier(25,20)(26,22)(25,24) 

\qbezier(0,44)(0,59)(0,70)
\qbezier(50,44)(50,59)(50,70)

\qbezier(0,70)(25,70)(50,70)
\qbezier(0,44)(25,44)(50,44)
\qbezier(5,24)(5,29)(15,34)
\qbezier(15,34)(25,39)(25,44)
\qbezier(25,24)(25,28)(20,31)
\qbezier(10,37)(5,40)(5,44)
\qbezier(44,20)(44,32)(44,44)


\put(100,53){$\beta$}

\qbezier(86,70)(86,78)(86,86) 
\qbezier(105,70)(105,78)(105,86)
\qbezier(124,70)(124,78)(124,86)

\qbezier(85,20)(84,22)(85,24) 
\qbezier(105,20)(106,22)(105,24) 

\qbezier(80,44)(80,59)(80,70)
\qbezier(130,44)(130,59)(130,70)

\qbezier(80,70)(105,70)(130,70)
\qbezier(80,44)(105,44)(130,44)

\qbezier(85,24)(85,28)(90,31)
\qbezier(100,37)(105,40)(105,44)
\qbezier(105,24)(105,29)(95,34)
\qbezier(95,34)(85,39)(85,44)

\qbezier(124,20)(124,32)(124,44)


\put(180,53){$\beta$}

\qbezier(166,70)(166,78)(166,86) 
\qbezier(185,70)(185,78)(185,86)
\qbezier(204,70)(204,78)(204,86)
\qbezier(160,44)(160,59)(160,70)
\qbezier(210,44)(210,59)(210,70)

\qbezier(160,70)(185,70)(210,70)
\qbezier(160,44)(215,44)(210,44)

\qbezier(166,20)(166,32)(166,44) 
\qbezier(185,20)(185,32)(185,44)
\qbezier(204,20)(204,32)(204,44)





\put(10,37){\tiny{$0$}}
\put(25,37){\tiny{$0$}}
\put(46,37){\tiny{$0$}}

\put(90,37){\tiny{$0$}}
\put(105,37){\tiny{$0$}}
\put(125,37){\tiny{$0$}}

\put(168,37){\tiny{$s$}}
\put(185,37){\tiny{$d-s$}}
\put(206,37){\tiny{$0$}}


\put(5,5){\small{$L_{+}=\widehat{\beta \sigma_1}$}}
\put(80,5){\small{$L_{-}=\widehat{\beta {\sigma_1}^{-1}}$}}
\put(165,5){\small{$L_s=\widehat{\beta t_1^st_2^{d-s}}$}}

\end{picture}}
\caption{The framed links in the skein relation in open braid form.}
\label{phils}
\end{figure}

Restricting now to the case of classical links, which can be seen as framed links with all framings zero, we obtain from $\Phi_{d,D}$ invariants for classical knots and links. We shall denote these invariants by $\Theta_{d,D}$. We have:

\begin{equation}\label{thinv}
\Theta_{d,D}(q, \lambda_D)(\widehat{\alpha}) := \left ( \frac{1 - \lambda_D}{\sqrt{\lambda_D} (q-q^{-1})E_D} \right)^{n-1} \left ( \sqrt{\lambda_D} \right)^{\varepsilon(\alpha)} {\rm tr}_{d,D}(\delta ( \alpha ) ) ,
\end{equation}
where $\widehat{\alpha}$ is the closure of $\alpha \in B_n$, $\varepsilon ( \alpha)$ is as in Theorem~\ref{frinvthm} and $\delta$ denotes the restriction of $\gamma$ to $\mathbb{C}(q)B_n$. 

One would like to compare the invariants $\Theta_{d,D}$ to other known invariants of classical knots and links, and especially to the Homflypt polynomial since the algebra ${\rm Y}_{d,n}(q)$ is a generalization of the algebra ${\rm H}_n(q)$. Note that the family of invariants $\left \{ \Theta_{d,D} \right \}$ includes the Homflypt polynomial,  since for $d=1$ which is the case where all framings are zero, the algebras ${\rm H}_n(q)$ and ${\rm Y}_{1,n}(q)$ coincide. Moreover, in this case the traces $\tau$, ${\rm tr}_1$ and ${\rm tr}_{1, \{0\}}$ all coincide. Hence, for any classical braid $\alpha \in B_n$ we have:
\[
P(q, \lambda_{\rm H})(\widehat{\alpha} ) = \Theta_{1, \{0 \}} (\widehat{\alpha}) = \left (  \frac{1 - \lambda_{\rm H}}{ \sqrt{\lambda_{\rm H}}(q -q^{-1})} \right)^{n-1} \left ( \sqrt{\lambda_{\rm H}} \right)^{\varepsilon(\alpha)} {\rm tr}_{1, \{0\}} (\delta (\alpha)).
\]
For $d>1$, the two algebras coincide only in the cases where $q = \pm 1$ and $E_D=1$ \cite[Theorem 5]{ChLa}. However, there are no algebra homomorphisms connecting the algebras and the traces \cite{ChLa}, so as to compare the invariants $\Theta_{d,D}$ and $P$ algebraically. Further, the skein relation of $\Phi_{d,D}$ has no topological interpretation in the case of classical knots and links because it introduces framings. This makes it very difficult to compare the invariants $\Theta_{d,D}$ to $P$ using diagrammatic methods. The problem of comparing the invariants $\Theta_{d,D}$ to the Homflypt polynomial has been an open problem for a while and eventually it has been solved in \cite{ChJuKaLa}, as we will explain in the next section.

\section{Identifying the classical link invariants from ${\rm Y}_{d,n}(q)$}\label{yhinvsec}
In this section we will focus on the classical knot and link invariants $\Theta_d$ that are derived from the algebras ${\rm Y}_{d,n}(q)$ and we shall see that they are {\it not topologically equivalent to the Homflypt polynomial}. 
\subsection{{\it The specialized trace for classical links}}\label{brsec} We start by considering a classical link as a closure of a braid $\alpha \in B_n$. By mapping the braid group in the Yokonuma-Hecke algebra via the natural homomorphism $\delta: \mathbb{C}(q)B_n \longrightarrow {\rm Y}_{d,n}(q)$ that sends $\sigma_i \mapsto g_i$, we observe that the image of any $\alpha \in B_n$ through $\delta$ involves only the braiding generators and the elements $e_i$, but not directly the framing generators. We denote by $ {\rm Y}_{d,n}^{({\rm br})}(q):=\delta(\mathbb{C}(q)B_n)$. We note further that from the quadratic equation \eqref{yqeq} we have the following:
\begin{equation}\label{gibreq1}
(q-q^{-1}) e_i g_i = g_i^2 - 1 \, \in \,  {\rm Y}_{d,n}^{({\rm br})}(q).
\end{equation}
This leads to the equation \cite[Proposition 4.1]{ChJuKaLa}:
\begin{equation}\label{gibreq2}
(q-q^{-1}) e_i = g_i^3 - g_i - (q-q^{-1})^2 e_i g_i \, \in \,  {\rm Y}_{d,n}^{({\rm br})}(q).
\end{equation}
From the above we deduce that:
\[
g_i^{-1} = g_i - (q-q^{-1})e_i \in  {\rm Y}_{d,n}^{({\rm br})}(q).
\]

Note that by combining \eqref{gibreq1} and \eqref{gibreq2} we obtain:
\begin{equation}\label{eibr}
e_i =\frac{1}{q-q^{-1}} (g_i^3 - g_i )- (g_i^2 - 1)\in  {\rm Y}_{d,n}^{({\rm br})}(q).
\end{equation}
This means that the algebra ${\rm Y}_{d,n}^{({\rm br})}(q)$ coincides with the subalgebra that is generated by the elements $g_1 , \ldots , g_{n-1}, e_1 , \ldots , e_{n-1}$. We thus have the following:
\begin{proposition}[{\cite[Proposition 4.1 and Remark 4.2]{ChJuKaLa}}]
The image of the $\mathbb{C}(q)$-algebra homomorphism $\delta$ is the subalgebra ${\rm Y}_{d,n}^{({\rm br})}(q)$ of ${\rm Y}_{d,n}(q)$, generated by $g_1,  \ldots , g_{n-1}$.
\end{proposition}

Observe now that when computing the specialized trace ${\rm tr}_{d,D}$ on $\alpha \in B_n$, the framing generators appear only while applying the quadratic relation \eqref{yqeq} or the inverse relation \eqref{ginv} and then only in the form of the idempotents $e_i$. For this reason the rule involving the framing generators of the specialized trace ${\rm tr}_{d,D}$ can be substituted by two new rules that involve only the idempotents $e_i$. We have the following:
\begin{theorem}[{\cite[Theorem 4.3]{ChJuKaLa}}]\label{sptrwei}
Let $m\in \{1, \ldots , d \}$ and set $E_m := \frac{1}{m}$. Let $z$ be an indeterminate over $\mathbb{C}(q)$. There exists a unique linear Markov trace
\[
{\rm tr}_{d,m}: \bigcup_{n \geq 0}{\rm Y}_{d,n}^{(\rm br)}(q) \longrightarrow \mathbb{C}(q)[z]
\]
defined inductively on ${\rm Y}_{d,n}^{(\rm br)}(q)$, for all $n \geq 0$, by the following rules:
\[
\begin{array}{lrcll}
(i) & {\rm tr}_{d,m} (ab)& = & {\rm tr}_{d,m}(ba) & a, b \in {\rm Y}_{d,n}^{(\rm br)}(q)\\
(ii) & {\rm tr}_{d,m} (\mathbf{1}_{n+1})& = & 1 &\\
(iii) & {\rm tr}_{d,m} (ag_n)& = & z\, {\rm tr}_{d,m}(a) & a \in {\rm Y}_{d,n}^{(\rm br)}(q)\\
(iv) & {\rm tr}_{d,m} (ae_n)& = & E_m\, {\rm tr}_{d,m}(a) & a \in {\rm Y}_{d,n}^{(\rm br)}(q)\\
(iv) & {\rm tr}_{d,m} (ae_ng_n)& = & z\, {\rm tr}_{d,m}(a) & a \in {\rm Y}_{d,n}^{(\rm br)}(q) ,\\
\end{array}
\]
where $\mathbf{1}_{n+1}$ denotes the unit in ${\rm Y}_{d,n+1}(q)$. For all $a \in \bigcup_{n \geq 0} {\rm Y}_{d,n}^{(\rm br)}(q)$, we have that ${\rm tr}_{d,m} (a ) = {\rm tr}_{d,D} (a)$ where $D$ is any subset of $\mathbb{Z}/d\mathbb{Z}$ such that $\vert D \vert = m$. Note that, in this case $E_m = E_D$.
\end{theorem}

 As proved in \cite{ChJuKaLa} the invariants $\Theta_{d,D}$ do not depend on the sets $D$, so the notation was simplified. More precisely, Theorem~\ref{sptrwei} implies that the specialized trace ${\rm tr}_{d,D}$ on classical knots and links depends only on $|D|$ and not on the solution $X_D$ of the ${\rm E}$-system. Further,  by results in \cite{ChJuKaLa}, for $d$, $d^\prime$ positive integers with $d\leq d^\prime$, we have $\Theta_{d,D} = \Theta_{d^\prime, D^\prime}$ as long as $|D|= |D^\prime|$. We deduce that, if $|D^\prime| = d$, then $\Theta_{d^\prime, D^\prime} = \Theta_{d, \mathbb{Z}/d\mathbb{Z}}$. Therefore, the invariants $\Theta_{d,D}$ can be parametrized by the natural numbers, setting $\Theta_d :=\Theta_{d, \mathbb{Z}/d\mathbb{Z}}$ for all $d \in \mathbb{Z}_{>0}$ \cite[Proposition~4.6]{ChJuKaLa}. 

For the rest of the paper $D$ will always be $\mathbb{Z}/d\mathbb{Z}$, implying that $E_D = 1/d$. In order not to confuse the reader, we will keep on using our initial notation for $E_D$ and $\lambda_D$  as well as for the traces ${\rm tr}_d$ and ${\rm tr}_{d,D}$. 
\smallbreak
We shall proceed now with the comparison of the invariants $\Theta_d$ to the Homflypt polynomial $P$.  As it turned out \cite{ChJuKaLa}, the behaviour of $\Theta_d$ depends on whether it is applied on knots or on links.  A braid $\alpha \in B_n$ closes to a knot if and only if $\mu(\alpha)$ is an $n$-cycle in $S_n$, where $\mu$ denotes the natural surjection from $B_n$ to $S_n$. This allows us to treat the case of knots and the case of links separately.

 \subsection{{\it The invariants $\Theta_d$ on classical knots}} The invariants $\Theta_d$ are topologically equivalent to the Homflypt polynomial for the case of knots. In order to prove this we shall need first the following proposition:
 \begin{proposition}[{\cite[Theorem~5.8]{ChJuKaLa}}]\label{tranfprop}  The transformation $z \mapsto z /E_D$ of the trace parameter $z$ of the Ocneanu trace $\tau$  corresponds to the transformation $\lambda_{\rm H} \mapsto \lambda_D$ on the Homflypt polynomial at variables $(q, \lambda_{\rm H})$.
\end{proposition}

\begin{proof}
The proof is a straightforward computation. We have that:
\[
\lambda_{\rm H} = \frac{ z/E_D  - (q-q^{-1})}{z/E_D} = \frac{z - (q-q^{-1}) E_D} {z} = \lambda_D
\]
\end{proof}

Next, the specialized trace ${\rm tr}_{d,D}$ has to be compared to the Ocneanu trace $\tau$. Indeed, in \cite[Proposition~5.6]{ChJuKaLa} it was proved that for the case of braids that close to knots, the trace functions ${\rm tr}_{d,D}$ and $\tau$ are connected by the following relation:
\begin{equation}\label{trmultocn}
{\rm tr}_{d,D} ( \alpha) = E_D^{n-1} \tau(\alpha).
\end{equation}
So, by choosing the trace parameter of $\tau$ to be $z/E_D$ and by utilizing Proposition~\ref{tranfprop} and equation~\eqref{trmultocn}, one obtains the following:

 \begin{theorem}[{\cite[Conjecture]{ChmJaKaLa} and \cite[Theorem~5.8]{ChJuKaLa}}]\label{thmcoinchom1}
Let $X_D$ be a solution of the ${\rm E}$-system.
For any $\alpha \in B_n$ such that $\widehat{\alpha}$ is a knot,
\[
\Theta_d(q,z)(\widehat{\alpha}) = \Theta_1 (q, z/ E_D) (\widehat{\alpha}) = P(q, z /E_D) (\widehat{\alpha}).
\]

\end{theorem}

\subsection{{\it The invariants $\Theta_d$ on classical links}} We shall study now the behaviour of $\Theta_d$ on braids whose closure is a link with at least two components. If the link $L$ is split, that is $L = L_1 \sqcup L_2 \sqcup \ldots \sqcup L_m$, where $L_1, \ldots , L_m$ are links, by the multiplicative property of the invariants $\Theta_d$ we have that \cite[Proposition~3.3]{ChmJaKaLa}:
\[
\Theta_d(L) = \left (  \frac{1 - \lambda_D}{\sqrt{\lambda_D} (q-q^{-1})E_D} \right)^{m-1} \Theta_d ( L_1) \ldots \Theta_d(L_m) .
\]
Thus, one needs to examine only non-split links. For links that are disjoint unions of $k$ knots, which is a special case of a split link, an analogous relation to equation~\ref{trmultocn} holds:
\begin{equation}\label{trmultocngen}
{\rm tr}_{d,D} (\alpha) = E_D^{n-k} \tau (\alpha) ,
\end{equation}
which leads to the following result:
\begin{theorem}[{\cite[Theorem~6.2]{ChJuKaLa}}]\label{thmcoinchom2}
For any $\alpha \in B_n$ such that $\widehat{\alpha}$ is a disjoint union of $k$ knots,
\[
\Theta_d(q,z)(\widehat{\alpha}) = E_D^{1-k} \Theta_1 (q, z/ E_D) (\widehat{\alpha}) = E_D^{1-k}  P(q, z /E_D) (\widehat{\alpha}).
\]
\end{theorem}

However, this result does not hold for the general case of links. For example, using the transformation $ z \mapsto z/ E_D$, we have the following for the Hopf link, $H=\widehat{\sigma_1^2}$:
\[
\tau(\sigma_1^2) = 1 + (q-q^{-1}) z/E_D
\]
and
\[ 
{\rm tr}_{d,D} ( \sigma_1^2) = {\rm tr}_{d,D} (1+ (q-q^{-1})e_1g_1) = 1+ (q-q^{-1})z = 1 -E_D + E_D\, \tau(\sigma_1^2)
\]
and thus
\[ 
\Theta_d (q,z) (H) \neq P(q, z/E_D) (H).
\]
The general case was studied in \cite{ChJuKaLa} with the use of a special skein relation for the invariants $\Theta_d$, which can only be applied on crossings between different components of a link. This skein relation for the invariants $\Theta_d$ was found via the invariants $\Phi_{d,D}$. More precisely, recall the skein relation \eqref{phiskein}:
\[
\frac{1}{\sqrt{\lambda_D}}\, \Phi_{d,D} (L_+) - \sqrt{\lambda_D} \, \Phi_{d,D} (L_-) = \frac{(q-q^{-1} )}{d} \sum_{s=0}^{d-1} \Phi_{d,D}(L_s) ,
\]
 where, for $\beta \in \mathcal{F}_n$, $L_+ = \widehat{\beta \sigma_i}$ , $L_{-} = \widehat{\beta \sigma_i^{-1}}$ and $L_s = \widehat{\beta t_i^{s} t_{i+1}^{-s}}$. If the strands $i$ and $i+1$ of $L_+$ belong to different components at the region of the crossing, then these strands must belong to the same component of $\beta$ and of $L_s$, otherwise the application of $\sigma_i$ would transfer them to the same component of $L_+$. As a consequence, in $L_s$ the framing of the $i$-th strand is added to the framing $(i+1)$-st and, since they add up to zero, the link $L_s$ can also be represented by the braid $\beta$. Namely, we have that:
\[
\frac{1}{d} \sum_{s=0}^{d-1} \Phi_{d,D}(L_s) = \frac{1}{d} \sum_{s=0}^{d-1} \Phi_{d,D}(\widehat{\beta}) =  \Phi_{d,D}(\widehat{\beta}) =  \Phi_{d,D}(L_0).
\]
Hence the skein relation~\eqref{phiskein} reduces to the following:
\[
\frac{1}{\sqrt{\lambda_D}} \, \Phi_{d,D}(L_+) - \sqrt{\lambda_D}\, \Phi_{d,D} (L_{-}) = (q-q^{-1})\,  \Phi_{d,D}(L_0).
\]
Restricting now to the case of classical links we obtain from the above the following:
\begin{theorem}[{\cite[Proposition~6.8]{ChJuKaLa}}]\label{thetaskeinthm}
The following skein relation holds for $\Theta_d$ only on crossings between different components:
\begin{equation}\label{thetaskein}
\frac{1}{\sqrt{\lambda_D}} \, \Theta_d(L_+) - \sqrt{\lambda_D}\, \Theta_d (L_{-}) = (q-q^{-1})\,  \Theta_d(L_0).
\end{equation}
\end{theorem}
So, in order to compute the invariant $\Theta_d$ on an $\ell$-component link, we apply the skein relation~\eqref{thetaskein} on $L$ and we unlink its components one-by-one. At the end of this procedure we obtain a disjoint union of knots with up to $\ell$ components. Thus, $\Theta_d(L)$ is written as a linear combination of values of $\Theta_d$ on the disjoint union of links. For $k=1, \ldots , \ell$, let $\mathcal{N}(L)_k$ denote the set of all disjoint unions of $k$ knots appearing in this linear combination and let $\mathcal{A}:=\mathbb{Q}[q^{\pm1}, \sqrt{\lambda_D}^{\pm1}]$. We then have the following result:

\begin{theorem}[{\cite[Theorem~6.16]{ChJuKaLa}}]\label{thmellink}
For any $\ell$-component link $L$, the value $\Theta_d(L)$ is an $\mathcal{A}$-linear combination of $P(L)$ and the values of $P$ on disjoint unions of knots obtained by the skein relation:
\[
\Theta_d(L) = \sum_{k=1}^{\ell} E_D^{1-k} \sum_{\widehat{\alpha} \in \mathcal{N}(L)_k} c(\widehat{\alpha}) P (\widehat{\alpha}) = P(L) + \sum_{k=2}^{\ell} \left(E_D^{1-k} - 1 \right) \sum_{\widehat{\alpha}\in \mathcal{N}(L)_k} c(\widehat{\alpha}) P (\widehat{\alpha}).
\]
\end{theorem}
In \cite{ChJuKaLa} 89 pairs of links that are equivalent through the Homflypt polynomial are considered. We shall call such pairs of links {\it $P$-equivalent}. These pairs are different links even if they are considered as unoriented links. Out of these 89 $P$-equivalent pairs of links, 83 were still equivalent through the $\Theta_d$ invariants, for generic $d$. We shall call such pairs {\it $\Theta_d$-equivalent}. However,  six pairs of 3-component $P$-equivalent links were found that are not $\Theta_d$-equivalent, for every $d$. For these pairs the authors of \cite{ChJuKaLa} computed the differences of the polynomials:
{\small \begin{align*}
&\Theta_d(L11n358\{0,1\})-\Theta_d(L11n418\{0,0\}) = \frac{(E_D-1) (\lambda_D -1) (q-1)^2 (q+1)^2 \left(q^2-\lambda_D \right)
   \left(\lambda_D  q^2-1\right)}{E_D \lambda_D^4 q^4}
\\
&\Theta_d(L11a467\{0,1\})-\Theta_d(L11a527\{0,0\}) = \frac{(E_D-1) (\lambda_D -1) (q-1)^2 (q+1)^2 \left(q^2-\lambda_D \right)
   \left(\lambda_D q^2-1\right)}{E_D \lambda_D^4 q^4}
\\
&\Theta_d(L11n325\{1,1\})-\Theta_d(L11n424\{0,0\}) = -\frac{(E_D-1) (\lambda_D -1) (q-1)^2 (q+1)^2 \left(q^2-\lambda_D \right)
   \left(\lambda_D q^2-1\right)}{E_D \lambda_D ^3 q^4}
\\
&\Theta_d(L10n79\{1,1\})-\Theta_d(L10n95\{1,0\}) = \frac{(E_D-1) (\lambda_D -1) (q-1)^2 (q+1)^2 \left(\lambda_D +\lambda_D
   q^4+\lambda_D  q^2-q^2\right)}{E_D \lambda_D^4 q^4}
\\
&\Theta_d(L11a404\{1,1\})-\Theta_d(L11a428\{0,1\}) = \frac{(E_D-1) (\lambda_D -1) (\lambda_D+1) (q-1)^2 (q+1)^2
   \left(q^4-\lambda_D  q^2+1\right)}{E_D q^4}
\\  
&\Theta_d(L10n76\{1,1\})-\Theta_d(L11n425\{1,0\}) =  \frac{(E_D-1) (\lambda_D -1) (\lambda_D+1) (q-1)^2 (q+1)^2}{E_D \lambda_D^3 q^2}.
\end{align*}}

\noindent Note that the factor $(E_D -1)$ that is common in all six pairs suggests that the pairs have the same Homflypt polynomial, since for $E_D=1$ the difference collapses to zero. The above analysis leads to the following exciting result:
\begin{theorem}[{\cite[Theorem 7.1]{ChJuKaLa}}]\label{notopoeq}
The invariants $\Theta_d$ are not topologically equivalent to the Homflypt polynomial for any $d\geq 2$.
\end{theorem}

\begin{remark}\rm
 It is worth adding here that the invariants $\Theta_d$ are not topologically equivalent to the Kauffman polynomial \cite{kau}, since there is at least one pair of knots which are distinguished by the Homflypt polynomial but not by the Kauffman polynomial.
 \end{remark}

\begin{remark} \label{genthetarem}\rm
In \cite[Section 8]{ChJuKaLa} the family of invariants $\left \{ \Theta_{d} \right \}$ has been generalized to a new 3-variable invariant $\Theta(q,\, \lambda, \, E)$, for $E\in \mathbb{C}$. In particular, $\Theta$ specializes to the Homflypt polynomial for $E=1$ and is stronger than the Homflypt polynomial. The invariant $\Theta$ satisfies also the special skein relation of the invariants $\Theta_d$ and, alike $\Theta_d$, can be completely defined using just the special skein relation and its values on disjoint unions of knots \cite{ChJuKaLa}. A diagrammatic proof of the well-definedness of $\Theta$ is given in \cite{kaula}. Finally, for properties of the invariants $\Theta_d$ and $\Theta$ the reader is referred to \cite{ChmJaKaLa, ChJuKaLa}
\end{remark}

\section{Framization of the Temperley-Lieb algebra of type $A$}\label{frtl}
In this section we first present the framization of the Temperley-Lieb algebra, ${\rm FTL}_{d,n}$, as one of the possible quotients of the Yokonuma-Hecke algebra. We then recall the necessary and sufficient conditions for the Markov trace ${\rm tr}_d$ to pass through to the quotient algebra ${\rm FTL}_{d,n}$ and we give the related link invariants.
\subsection{{\it The first attempts}}The definition of a framization of the Temperley-Lieb algebra using the techniques of \cite{jo, jula} has been an open problem for some time. Our first attempt was in \cite{gojukola} where the {\it Yokonuma-Temperley-Lieb algebra} was defined, for $n \geq 3$, as a quotient of the algebra ${\rm Y}_{d,n}(u)$ (recall Remark~\ref{oldquadr}) over the two-sided ideal generated by the same expression as in the classical case, that is:
\[
{\rm YTL}_{d,n}(u) := \frac{{\rm Y}_{d,n}(u)} { \langle 1+ g_1 +g_2 + g_1 g_2 + g_2 g_1 + g_1 g_2 g_1 \rangle} .
\]

 The necessary and sufficient conditions so that the trace ${\rm tr}_d$ on ${\rm Y}_{d,n}(u)$ passes to the quotient ${\rm YTL}_{d,n}(u)$ proved to be too restrictive \cite[Theorem~6]{gojukola}. Namely, only trivial solutions of the ${\rm E}$-system would qualify, and, as a consequence, the resulting invariants for framed knots and links were not topologically interesting. More precisely, basic pairs of framed links were not distinguished. However, for the case of classical links we were able to recover the Jones polynomial.
 
The second attempt was made in \cite{gojukola2} with {\it the Complex Reflection Temperley-Lieb} algebra. This time we defined the quotient algebra with a two-sided ideal of ${\rm Y}_{d,n}(u)$ that was analogous to the classical case. To be more precise, we return to the discussion of Section~\ref{prelim} and we consider the Iwahori-Hecke algebra as a $u$-deformation of $\mathbb{C}S_n$. We note also that the underlying group of the defining ideal of  the algebra ${\rm TL}_n(u)$ is $S_3$. In this context,  the Yokonuma-Hecke algebra is seen as a $u$-deformation of $\mathbb{C}C_{d,n}$ and so we consider, for the definition of the quotient algebra, the two-sided ideal whose underlying group is $C_{d,3}$, which is completely analogous to the classical case. Thus, for $n \geq 3$, we define the following quotient of the algebra ${\rm Y}_{d,n}(u)$:
\[
{\rm CTL}_{d,n}(u) := \frac{{\rm Y}_{d,n}(u)} { \left \langle \displaystyle \sum_{a,\, b,\,c \,\in \mathbb{Z}/d\mathbb{Z}} t_1^a \, t_2^b \, t_3 ^c (1+ g_1 +g_2 + g_1 g_2 + g_2 g_1 + g_1 g_2 g_1) \right \rangle} .
\]
\begin{remark} \rm
The denomination Complex Reflection Temperley-Lieb algebra has to do with the fact that the underlying group of ${\rm CTL}_{d,n}(u)$ is isomorphic to the complex reflection group $G(d,1,3)$.
\end{remark}
The necessary and sufficient conditions so that the trace ${\rm tr}_d$ passes to the quotient algebra ${\rm CTL}_{d,n}(u)$ proved to be too relaxed, especially on the trace parameters $x_i$ \cite[Theorem~7]{gojukola2}. So, in order to define link invariants from the algebras ${\rm CTL}_{d,n}(u)$, the ${\rm E}$-condition must be imposed on the $x_i$'s.  Further, as we showed in \cite[Proposition~10]{gojukola2}, the invariants that are derived from ${\rm CTL}_{d,n}(u)$ coincide either with those from ${\rm Y}_{d,n}(u)$ or with those from another quotient of ${\rm Y}_{d,n}(u)$ that will be discussed next. The algebra ${\rm CTL}_{d,n}(u)$ proved to be unnecessarily large for our topological purposes. Indeed, it is the largest one of the three quotients of ${\rm Y}_{d,n}(u)$ in discussion and, since we do not obtain any extra invariants,  it is discarded as a possible framization of the Temperley-Lieb algebra.

\subsection{{\it The Framization of the Temperley-Lieb algebra}} The discussion above indicated that the desired framization of the Temperley-Lieb algebra, for our topological purposes, could be an intermediate algebra between the quotient algebras ${\rm YTL}_{d,n}(u)$ and ${\rm CTL}_{d,n}(u)$. We achieve this, by using for the defining ideal an intermediate subgroup that lies between  $S_3 $ and $C_{d,3}$. Indeed, by considering the following subgroup of $C_{d,3}$ (see \cite[Section~4.2]{gojukola2}):
\[
H_{d,3} := \langle t_1 t_{2}^{-1}, t_{2}t_{3}^{-1} \rangle \rtimes S_3 ,
\]
we define:
\begin{definition}[{\cite[Definition~5]{gojukola2}}]\label{ftldefine}\rm
For $n \geq 3$, the {\it Framization of the Temperley-Lieb algebra}, denoted  ${\rm FTL}_{d,n}(u)$, is defined as:
\[
{\rm FTL}_{d,n}(u) := \frac{{\rm Y}_{d,n}(u)}{\left \langle e_1 e_2 \big( 1 +  g_1 +g_2  +  g_1 g_2 + g_2 g_1  +  g_1 g_2 g_1 \big) \right \rangle } .
\]
\end{definition}

\begin{remark} \rm
 M. Chlouveraki and G. Pouchin studied extensively the representation theory of all three quotient algebras of ${\rm Y}_{d,n}(u)$ that we presented so far. Further, they provided linear bases for all three of them  and computed their dimensions \cite{ChPou, ChPou2}.
 \end{remark}

We now move on to the determination of the necessary and sufficient conditions so that the trace ${\rm tr}_d$ passes to the quotient algebra ${\rm FTL}_{d,n}(u)$.  Since the defining ideal of  ${\rm FTL}_{d,n}(u)$ is principal and by the linearity of ${\rm tr}_d$, we have that ${\rm tr}_d$ passes to  ${\rm FTL}_{d,n}(u)$  if and only if: 
\begin{equation}\label{trwg12}
{\rm tr}_d(\mathfrak{m}\, r_{1,2})=0,
\end{equation}
where $r_{1,2}:= e_1 e_2 \big( 1 +  g_1 +g_2  +  g_1 g_2 + g_2 g_1  +  g_1 g_2 g_1 \big)$, for all  monomials $\mathfrak{m}$ in the inductive basis of ${\rm Y}_{d,n}(u)$. Our approach to proving the above statement in \cite{gojukola2} was to work first for the case $n=3$ and then to generalize the result by using induction on $n$. By \eqref{yhbasel} the elements in the inductive basis of ${\rm Y}_{d,3}(u)$  are of the following forms:
\begin{equation}\label{basicwords}
t_1^{a}t_2^{b}t_3^c,  \quad t_1^{a}g_1t_1^{b}t_3^c, \quad t_1^{a}t_2^{b}g_2g_1t_1^c, \quad  t_1^{a}t_2^{b}g_2t_2^c,\quad t_1^{a}g_1t_1^{b} g_2t_2^c, \quad t_1^{a}g_1t_1^{b} g_2g_1t_1^c,
\end{equation}
where $0\leq a,b,c \leq d-1$. Substituting each one of the six elements of \eqref{basicwords} into \eqref{trwg12} we obtain the following system of equations:
\begin{equation} \label{systftl}
(u+1)z^2x_{m} + (u+2)z \, E_D^{(m)}  +{\rm tr}_d(e_1^{(m)}e_2) =0 \quad \mbox{for } 0 \leq m \leq d-1.
\end{equation}
The necessary and sufficient conditions for ${\rm tr}_d$ to pass to ${\rm FTL}_{d,n}(u)$ emerged after solving the system \eqref{systftl} in $\mathbb{C}C_d$ by making use of the methods of Theorem~\ref{esysol} and thus we obtained the following:
\begin{theorem}[{\cite[Theorem~6]{gojukola2}}]\label{akthmgen}
The trace ${\rm tr}_d$ passes to ${\rm FTL}_{d,n}(u)$ if and only if the parameters of the trace ${\rm tr}_d$ satisfy: 
{\footnotesize \[
x_k = -z \left(\sum_{m\in {\rm Sup}_1}\chi_{ m}(t^{k}) + (u+1)\sum_{m\in {\rm Sup}_2}\chi_{ m}(t^{k}) \right)
\quad \text{and}\quad  
z=-\frac{1}{ \vert {\rm Sup_1}\vert + (u+1)\vert {\rm Sup_2}\vert  },
\]
}
where ${\rm Sup}_1\cup \rm{Sup}_2$ (disjoint union) is the support of the Fourier transform of $x$ and $x$ is the complex function on $C_d$ that maps $0$ to $1$ and $k$ to the trace parameter  $x_k$.

\end{theorem}
We now have the following Corollary \cite[Corollary~3]{gojukola2}.
\begin{corollary}\label{maincor}
 In the case where one of the sets ${\rm Sup}_1$ or ${\rm Sup}_2$ is the empty set, the values of the $x_k$'s comprise a solution of the ${\rm E}$-system. More precisely,  if ${\rm Sup}_1$ is the empty set, the $x_k$'s are the solutions of the   ${\rm E}$-system parametrized by ${\rm Sup}_2$ and $z= -1/(u+1)\vert{\rm Sup}_2\vert$. If ${\rm Sup}_2$ is the empty set, then $x_k$'s are the solutions of the ${\rm E}$-system parametrized by ${\rm Sup}_1$ and 
$z= -1/ \vert {\rm Sup}_1 \vert$. In this way, we obtain all solutions of the ${\rm E}$-system.
\end{corollary}

Since for defining classical link invariants only the cardinal $\vert D \vert $ of a parametrizing set $D$ of a solution of the ${\rm E}$-system is needed, the solutions described in Corollary~\ref{maincor} cover all possibilities.

\begin{remark}\label{disc}\rm
 We do not take into consideration the case where $z=-\frac{1}{|D|}$, since important topological information is lost. For example, the trace ${\rm tr}_{d,D}$ gives the same value for all even (resp. odd) powers of the $g_i$'s, for $m \in \mathbb{Z}_{>0} $ \cite{jula}:
\begin{equation*}\label{evenpower}
{\rm tr}_{d,D}(g_i^m) = \left( \frac{u^m -1}{u+1} \right )z + \left( \frac{u^m -1}{u+1} \right )\frac{1}{\vert D \vert} +1 \qquad \text{if } m \text{ is even}
\end{equation*} 
and
\begin{equation*}\label{oddpower}
{\rm tr}_{d,D}(g_i^m) = \left( \frac{u^m -1}{u+1} \right ) z + \left( \frac{u^m -1}{u+1} \right ) \frac{1}{\vert D \vert} - \frac{1}{\vert D \vert} \qquad \text{if } m \text{ is odd,}
\end{equation*}so the corresponding knots and links are not distinguished.
\end{remark}

\begin{remark} \rm
As mentioned in Section~\ref{invdefsec} the trace parameters $x_i$ satisfying the ${\rm E}$-system is the requirement for the definition of framed and classical link invariants from the algebras ${\rm Y}_{d,n}(u)$. With this in mind, we note that one of the reasons that makes the quotient algebra ${\rm FTL}_{d,n}(u)$ to stand out from the quotient algebras ${\rm YTL}_{d,n}(u)$ and ${\rm CTL}_{d,n}(u)$ is the fact that all solutions of the ${\rm E}$-system are included in the necessary and sufficient conditions of  Theorem~\ref{akthmgen}. Recall that in the case of ${\rm YTL}_{d,n}(u)$ we only have trivial solutions of the ${\rm E}$-system, while in the case of ${\rm CTL}_{d,n}(u)$ they have to be imposed on the conditions of the analogue to Theorem~\ref{akthmgen} for ${\rm CTL}_{d,n}(u)$.
\end{remark}

\smallbreak
For the remainder of this paper we will return to the presentation of ${\rm Y}_{d,n}$ with parameter $q$. This is because it makes computations considerably easier and thus we are able to compare the 1-variable invariants that we will construct below to the Jones polynomial. In this context, we note that we kept the original presentations, with parameter $u$, for the algebras ${\rm YTL}_{d,n}$ and ${\rm CTL}_{d,n}$ since they are discarded as potential framizations of the Temperley-Lieb algebra and so there is no need to study further their derived invariants.

\subsection{{\it Framed and classical link invariants from ${\rm FTL}_{d,n}(q)$}} 
In order to define link invariants on the level of the quotient algebra ${\rm FTL}_{d,n}(q)$ we only need to transform using \eqref{switch} the results of the previous section and then substitute them in the formulas that describe the invariants $\Phi_{d,D}(q, \lambda_D)$ and $\Theta_{d}(q, \lambda_D)$. Indeed, the  generator of the defining ideal of the algebra ${\rm FTL}_{d,n}(q)$ becomes: $e_1 e_2 (1+ q( g_1 + g_2) + q^2 (g_1 g_2 +g_2 g_1) + q^3 g_1 g_2 g_1)$. Let now $D$ be a non-empty set of $\mathbb{Z}/d\mathbb{Z}$ and let ${\rm Sup}_1 = \emptyset$ and $\vert {\rm Sup}_2 \vert = \vert D \vert $. Then, by Corollary~\ref{maincor}, Remark~\ref{disc} and \eqref{lambdad}, we obtain:
\[
z= - \frac{q^{-1}}{(q^2+1) |D|} \quad \mbox{and} \quad  \lambda_D = q^4 .
\]

\begin{definition}\label{framinvdef}\rm
Let $X_D$ be a solution of the ${\rm E}$-system, parametrized by the non-empty subset $D$ of $\mathbb{Z}/d\mathbb{Z}$ and let $z = -\frac{q^{-1}}{(q^2+1)|D|}$. We obtain from $\Phi_{d,D}(q, \lambda_D)$ the following 1-variable  invariant of framed links:
\[
 \vartheta_{d,D}(q)(\widehat{\alpha}): =  \left ( -\frac{1+q^2}{qE_D} \right)^{n-1} q^{2\varepsilon(\alpha)}  {\rm tr}_{d,D}\left(\gamma(\alpha)\right) = \Phi_{d,D}(q,q^4)(\widehat{\alpha}) ,
 \]
for any $\alpha \in \cup_{\infty}\mathcal{F}_n$. Further, in analogy to  the invariants $\Phi_{d,D}(q, \lambda_D)$, if we restrict to framed links with all framings zero, we obtain from  $\vartheta_{d,D}(q)$ an 1-variable invariant of classical links, namely:
\[
\theta_d(q)(\widehat{\alpha}) := \left ( -\frac{1+q^2}{qE_D} \right)^{n-1} q^{2\varepsilon(\alpha)} {\rm tr}_{d,D} (\delta(a)) = \Theta_{d}(q,q^4)(\widehat{\alpha}).
\]
\end{definition}

\subsection{{\it Identifying the classical link invariants from the algebra ${\rm FTL}_{d,n}(q)$}}
To conclude this section, we shall compare the classical link invariants $\theta_d(q)$ to the Jones polynomial. We first observe that for $d=1$ one recovers the Jones polynomial. Having now in mind the discussion of Section~\ref{yhinvsec} we distinguish two cases, one for the case of knots and disjoint unions of knots and one for the case of links. Recall that the invariants $\Theta_d$ are topologically equivalent to the Homflypt polynomial (see Theorem~\ref{notopoeq}) for the case of knots and disjoint unions of knots. By specializing the parameter $z$ of ${\rm tr}_{d,D}$ to the value $ -\frac{q^{-1}}{(q^2+1)|D|}$, where $| D | =d, $ in Theorems~\ref{thmcoinchom1} and \ref{thmcoinchom2} we observe that these properties of the invariants $\Theta_d$ are preserved on the level of the quotient algebra ${\rm FTL}_{d,n}(q)$ and the invariants $\theta_d$. Indeed we have the following:
\begin{proposition}[{\cite[Proposition~11]{gojukola2}}]
The invariants $\theta_d$ are topologically equivalent to the Jones polynomial on knots and disjoint unions of knots.
\end{proposition}
Further, recall that the invariants $\Theta_d$, for $d>1$ are not topologically equivalent to the Homflypt polynomial on links. By specializing in Theorems~\ref{thetaskeinthm} and \ref{thmellink} the trace parameter $z$ of ${\rm tr}_{d,D}$ to the value $ -\frac{q^{-1}}{(q^2+1)|D|}$, one can deduce that this property holds also on the level of the quotient algebra ${\rm FTL}_{d,n}(q)$ and the invariants $\theta_d$. Furthermore, the special skein relation for $\theta_d$ is easily deduced from \eqref{thetaskein} by substituting $\lambda_D= q^4$. We thus have:
\begin{theorem}[{\cite[Theorem~9]{gojukola2}}]
For $d\in \mathbb{Z}_{>1}$, the invariants $\theta_d(q)$  for classical links are not topologically equivalent to the Jones polynomial. Further, the invariants $\theta_d(q)$ satisfy the following special skein relation:
\[
q^{-2}\, \theta_d\, (L_{+}) - q^2\, \theta_d\, (L_{-}) = (q-q^{-1})\, \theta_d\, (L_{0}),
\]
where the oriented links $L_{+}$, $L_{-}$, $L_{0}$ comprise a Conway triple involving a crossing between different components. 
\end{theorem}

\begin{remark} \rm
By substituting $\lambda_D$ by $q^4$ in the computations of $\Theta_d$ on the six pairs of $P$-equivalent links, we find that they are all still distinguished by $\theta_d$, namely:

{\small\begin{align*}
&\theta_d(L11n358\{0,1\})-\theta_d(L11n418\{0,0\}) = \frac{(1-E_D) (q-1)^5 (q+1)^5 (q^2+1)(q^2+q+1)(q^2-q+1)}{E_D\, q^{18}}
\\
&\theta_d(L11a467\{0,1\})-\theta_d(L11a527\{0,0\}) =  \frac{(1-E_D) (q-1)^5 (q+1)^5 (q^2+1)(q^2+q+1)(q^2-q+1)}{E_D\, q^{18}}
\\
&\theta_d(L11n325\{1,1\})-\theta_d(L11n424\{0,0\}) = \frac{(E_D-1) (q-1)^5 (q+1)^5 (q^2+1)(q^2+q+1)(q^2-q+1)}{E_D\, q^{14}}\\
&\theta_d(L10n79\{1,1\})-\theta_d(L10n95\{1,0\}) = \frac{(E_D-1) (q^2-1)^3 (q^8+2\,q^6+2\,q^4-1)}{E_D \, q^{18}}\\
&\theta_d(L11a404\{1,1\})-\theta_d(L11a428\{0,1\}) = \frac{(1-E_D) (q-1)^3(q+1)^3(q^2+1)(q^4+1)(q^6-q^4+1)}{E_D\, q^4}
\\  
&\theta_d(L10n76\{1,1\})-\theta_d(L11n425\{1,0\}) =  \frac{(E_D-1) (q-1)^3(q+1)^3(q^2+1)(q^4+1)}{E_D\, q^{10}}.\\
\end{align*}}
As mentioned earlier, for $E_D=1$ the invariants $\theta_d$ coincide with the Jones polynomial and the above six differences collapse to zero. 
\end{remark}

\section{Connections with the algebra of braids and ties and the partition Temperley-Lieb algebra}\label{furth}
In this section we point out the connection between ${\rm Y}_{d,n}^{(\rm br)}(q)$, the subalgebra of ${\rm Y}_{d,n}(q)$ that is generated only by the braiding generators $g_i$ (see Section~\ref{brsec}), and the {\it algebra of braids and ties} $\mathcal{E}_n(q)$. We  then will establish an analogous result between the {\it partition Temperley-Lieb algebra}, which is defined as a quotient of $\mathcal{E}_n(q)$ over an appropriate 2-sided ideal, and the subalgebra ${\rm FTL}_{d,n}^{(\rm br) }(q)$ of ${\rm FTL}_{d,n}(q)$ that is generated by the braiding generators $g_i$.
\subsection{{\it The algebra of braids and ties}} The algebra $\mathcal{E}_n(q)$ was first introduced in \cite{jubt} and its definition emerged as an abstraction of a non-standard presentation of the Yokonuma-Hecke algebra where the framing generators $t_i$ are left aside and the idempotents $e_i$ are used instead. Its defining relations are obtained by imposing the commuting relations of the braiding generators of the algebra ${\rm Y}_{d,n}(q)$ with the idempotents $e_i$. 

In order to avoid confusion with the algebra ${\rm Y}_{d,n}(q)$ and its quotients, we shall use from now on the notation $b_i$, $1 \leq i \leq n-1$, for the braiding generators of $\mathcal{E}_n(q)$ and $\epsilon_i$, $1 \leq i \leq n-1$, for its idempotent generators. In terms of generators and relations, the algebra $\mathcal{E}_n(q)$ is the $\mathbb{C}(q)$-algebra that is generated by the elements $b_1, \ldots, b_{n-1}, \epsilon_1,$ $\ldots,\epsilon_{n-1}$ that satisfy the following relations:

\begin{eqnarray}
b_i b_{i+1} b_i &=& b_{i+1} b_i b_{i+1} \label{btpres}\\
b_i b_j &=& b_j b_i \quad \mbox{for }  |i-j|>1\\
\epsilon_i \epsilon_j &=&\epsilon_j \epsilon_i \quad \mbox{for } |i-j|>1\\
\epsilon_i^2 &=& \epsilon_i\\
\epsilon_ib_i &=& b_i \epsilon_i\\
\epsilon_i b_j &=& b_j\epsilon_i \quad \mbox{for } |i-j|>1\\
\epsilon_i\epsilon_j b_i &=& b_i \epsilon_i \epsilon_j  = \epsilon_j b_i \epsilon_j \quad \mbox{for } |i-j|=1\\
\epsilon_i b_j b_i &= &b_j b_i \epsilon_j \quad \mbox{for } |i-j|=1\\
b_i^2& =  &1 +(q-q^{-1})\epsilon_ib_i \label{btpres1}.
\end{eqnarray}

\begin{remark}\label{btswitch} \rm
The original presentation of the algebra of braids and ties involves generators $\widetilde{b}_1, \ldots, \widetilde{b}_{n-1}$, $\epsilon_1, \ldots , \epsilon_{n-1}$ that satisfy all relations in the presentation above except for the quadratic relation, which is replaced by one with parameter $u$ instead of $q$, namely:
\[
(\widetilde{b}_i)^2 = 1 + (u-1)\epsilon_i + (u-1)\epsilon_i \widetilde{b}_i .
\]
In this paper we adopt the presentation with parameter $q$ that was used in \cite{ChJuKaLa, ERH}. By applying an analogous transformation as in Remark~\ref{oldquadr}, namely $b_i := \widetilde{b}_i +(q^{-1} -1) \epsilon_i \widetilde{b}_i$ (or equivalently $\widetilde{b}_i := b_i +(q-1) \epsilon_i b_i$) and choosing $u=q^2$, one can easily switch from the presentation given here to the original one and vice versa.
\end{remark}

In \cite{srh} a faithful tensorial representation for the algebra $\mathcal{E}_n(u)$ was constructed which was then used in order to classify its irreducible representations. In addition to that, a linear basis for $\mathcal{E}_n(u)$ was constructed which was later used in \cite{AiJu1} in order to define a linear Markov trace function on the algebra $\mathcal{E}_n(u)$. The introduction of the Markov trace led to the definition of 3-variable invariants for classical, singular and tied links \cite{AiJu1, AiJu2}.

In a recent development \cite[Theorem~8]{ERH}  J. Espinoza and S. Ryom-Hansen proved that the map:
\begin{equation}\label{btiso}
\begin{aligned}
 \phi: \mathcal{E}_n(q) &\longrightarrow {\rm Y}_{d,n}(q) \\
 b_i & \mapsto g_i \\
 \epsilon_i & \mapsto e_i 
 \end{aligned}
 \end{equation}
 is actually an injection when $d\geq n$ and thus $\mathcal{E}_n(q)$ is isomorphic to its image through $\phi$. Hence $\mathcal{E}_n(q)$ is isomorphic to the subalgebra of ${\rm Y}_{d,n}(q)$  that is generated by the $g_i$'s and the $e_i$'s. On the other hand, by \eqref{eibr}, the subalgebra  ${\rm Y}_{d,n}^{(\rm br)}(q)$, which is generated by the $g_i$'s, coincides with the subalgebra of ${\rm Y}_{d,n}(q)$  that is generated by the $g_i$'s and the $e_i$'s (cf. \cite[Remark~4.4]{ChJuKaLa}). 

\begin{remark}\label{isoprop} \rm
From the above, it follows immediately that for $d \geq n$, the algebra of braids and ties $\mathcal{E}_n(q)$ is isomorphic to the subalgebra ${\rm Y}_{d,n}^{\rm(br)}(q)$ of ${\rm Y}_{d,n}(q)$. 
\end{remark}

\begin{remark} \rm
An analogous injection exists between the algebras $\mathcal{E}_n(u)$ and ${\rm Y}_{d,n}(u)$ can be derived from \eqref{btiso}. This can be shown by considering the automorphisms $\eta$ and $\beta$ of ${\rm Y}_{d,n}(u)$ and $\mathcal{E}_n(u)$ respectively, which send the non-braiding generators $t_i$ to $t_i$, while on the braiding generators $\widetilde g_i$ (resp. $\widetilde b_i$) they are defined as follows:
\[
\eta (\widetilde{g}_i ) = g_i + (q-1)e_i g_i \quad \mbox{and} \quad \beta ( \widetilde{b}_i) = b_i + (q-1) \epsilon_i b_i. 
\]
Notice that the automorphisms $\eta$ and $\beta$ correspond to the  transformations that were discussed in Remarks~\ref{oldquadr} and \ref{btswitch} respectively. The map $\psi : \mathcal{E}_n(u) \rightarrow {\rm Y}_{d,n}(u)$ is an injection by considering the fact that: $\psi = \eta^{-1} \circ \phi \circ \beta$, where $\eta^{-1} (g_i ) = \widetilde{g}_i + (q^{-1} - 1) e_i \widetilde{g}_i$ (see \cite[Remark~3]{AiJu1}).
\end{remark}

\subsection{{\it The partition Temperley-Lieb algebra}} We turn now our attention to the {\it partition Temperley-Lieb algebra}. For $n\geq 3$, the partition Temperley-Lieb algebra, denoted by ${\rm PTL}_n(u)$, was introduced by Juyumaya in \cite{juptl} as a quotient of the algebra $\mathcal{E}_n(u)$ over the two-sided ideal that is generated by the elements:
\[
\epsilon_i \epsilon_j \widetilde{b}_{i,j} \quad \mbox{for all } i,j \mbox{ such that }  |i-j|=1,
\]
where 
\[
\widetilde{b}_{i,j}:=1 + \widetilde{b}_i +\widetilde{b}_j + \widetilde{b}_i \widetilde{b}_j + \widetilde{b}_j \widetilde{b}_i + \widetilde{b}_i \widetilde{b}_j \widetilde{b}_i.
\]
It can be easily shown that this ideal is in fact principal and that is generated by the single element: 
\[
\epsilon_1 \epsilon_2 \widetilde{b}_{1,2}.
\]

\noindent In terms of the presentation with parameter $q$, the partition Temperley-Lieb algebra ${\rm PTL}_n(q)$ is generated by the elements $b_1, \ldots , b_{n-1}$, $\epsilon_1 , \ldots , \epsilon_{n-1}$, subject to the relations \eqref{btpres}-\eqref{btpres1} together with the following defining relation:
\begin{equation}\label{ptldefrel}
\epsilon_i \epsilon_{i+1} \left [ 1+ q(b_i + b_{i+1} ) + q^2 (b_i b_{i+1} + b_{i+1} b_i ) + q^3 b_i b_{i+1} b_i \right ] =0, \quad 1\leq i \leq n-1.
\end{equation}

Let now ${\rm FTL}_{d,n}^{(\rm br)}(q)$ denote the quotient of the algebra ${\rm Y}_{d,n}^{(\rm br)}(q)$ over the two-sided ideal ${\langle e_1 e_2 g_{1,2} \rangle}$, where $g_{1,2}: = 1+ q(g_1 + g_2) + q^2( g_1 g_2 + g_2 g_1) + q^3 g_1 g_2 g_1$. That is:
\[
{\rm FTL}_{d,n}^{(\rm br)}(q) = \frac{ {\rm Y}_{d,n}^{(\rm br)}(q)}{{\langle e_1 e_2 g_{1,2} \rangle}}.
\]

 Recall that, by \eqref{eibr}, we have that the elements $e_i$ belong to ${\rm Y}_{d,n}^{(\rm br)}(q)$, for $1 \leq i \leq n-1$. This means that $e_1 e_2 g_{1,2}$ also belongs to ${\rm Y}_{d,n}^{(\rm br)} (q)$ and so the quotient algebra is well defined.

On the other hand, the quotient-algebra ${\rm FTL}_{d,n}^{(\rm br)}(q)$ is a subalgebra of ${\rm FTL}_{d,n}(q)$ and can be considered as the image of ${\rm Y}_{d,n}^{(\rm br)}(q)$ via the natural projection $p : {\rm Y}_{d,n}(q) \longrightarrow {\rm FTL}_{d,n}(q)$. That is:

\[
{\rm FTL}_{d,n}^{(\rm br)}(q)= p \left({\rm Y}_{d,n}^{(\rm br)}(q)\right) = (p \circ \delta) (\mathbb{C}(q)B_n).
\]

The following is a consequence of Remark~\ref{isoprop}.

\begin{proposition} For $d \geq n$, the partition Temperley-Lieb algebra ${\rm PTL}_{d,n}(q)$  is isomorphic to the algebra ${\rm FTL}_{d,n}^{(br)}(q)$.
\end{proposition}

\begin{proof}
Let $b_{i,i+1}:= 1 + q(b_i +b_{i+1}) + q^2 (b_i b_{i+1} + b_{i+1} b_i) + q^3 b_i b_{i+1} b_i$. The generator $\epsilon_1 \epsilon_2 b_{1,2}$ of  $\langle \epsilon_1 \epsilon_2 b_{1,2} \rangle$ which is a two-sided ideal in $\mathcal{E}_n(q)$ is mapped via \eqref{btiso} to the element $e_1 e_2 g_{1,2} \in {\rm Y}_{d,n}^{(\rm br)}(q)$, which is the generator of  the two-sided ideal $\langle e_1 e_2 g_{1,2} \rangle $ of ${\rm Y}_{d,n}^{(\rm br)} (q)$. By Remark~\ref{isoprop}, and since both ideals  are principal, one deduces that the ideals are also isomorphic. Hence, we have that:
\[
{\rm PTL}_n(q) =\frac{ \mathcal{E}_n(q)}{\langle \epsilon_1 \epsilon_2 b_{1,2} \rangle} \cong \frac{{\rm Y}_{d,n}^{(\rm br)}(q)}{\langle e_1 e_2 g_{1,2}\rangle} = {\rm FTL}_{d,n}^{(\rm br)}(q) 
\]
and thus the proof of the Proposition is concluded.
\end{proof}

 \begin{remark} \rm
The elements of the algebra of braids and ties topologically close to oriented tied links which were introduced by F. Aicardi and J. Juyumaya in \cite{AiJu2} and were obtained by a diagrammatic interpretation of the generators of the algebra $\mathcal{E}_n(u)$. A {\it tied link} can be visualized as a classical link where two points of the classical link may be tied via a spring that slides along the component(s) that is attached to. We have a special interest in the isomorphic algebra $\mathcal{E}_n(q)$. The introduction of the new quadratic relation revealed that the related 3-variable invariant $\Theta(q,\lambda, E)$ that was constructed in \cite{ChJuKaLa} (recall Remark~\ref{genthetarem}) satisfies the same special skein relation as $\Theta_d$, something that wasn't possible with the original presentation of $\mathcal{E}_n(u)$. Recall that the invariant $\Theta(q,\lambda, E)$  generalizes the classical link invariants $\Theta_d(q, \lambda_D)$ as well as the Homflypt polynomial \cite{ChJuKaLa, kaula}. 
 
 Similarly,  our interest in the algebra ${\rm PTL}_n(q)$ lies in the fact that it is related to the 2-variable invariant $\theta(\lambda, E)$, which is constructed in \cite{gola} and satisfies the same special skein relation as the invariants $\theta_d$. Analogously, the invariant $\theta(q, E)$ generalizes the invariants $\theta_d$ as well as the Jones polynomial.
 \end{remark}

\bibliography{bibliography.bib}{}
\bibliographystyle{siam}
\end{document}